\documentclass[12pt]{amsart}

\usepackage{dsfont}
\usepackage{amsmath,amssymb,amsfonts,graphics,amsthm}
\usepackage{bbm}
\usepackage{bm}
\usepackage{enumerate}

\usepackage{fullpage}

\usepackage{graphicx}

\usepackage[T1]{fontenc}
\usepackage[utf8]{inputenc}
\usepackage[bbgreekl]{mathbbol}

\newcommand{\C}{\mathbb{C}}

\newcommand{\cF}{\mathcal{F}}
\newcommand{\G}{\mathbb{G}}

\newcommand{\bGamma}{\mathbb{\Gamma}}
\newcommand{\HH}{\mathsf{H}}

\newcommand{\K}{\mathsf{K}}

\newcommand{\op}[1]{\operatorname{#1}}

\newtheorem{thm}{Theorem}[section]

\newtheorem{lem}[thm]{Lemma}
\newtheorem{prop}[thm]{Proposition}

\theoremstyle{definition}
\newtheorem{defn}[thm]{Definition}

\newtheorem{rem}[thm]{Remark}

\newtheorem{ex}[thm]{Example}

\newtheorem{introtheorem}{Theorem}

\numberwithin{equation}{section}

\begin{document}
\title{On quantum Cayley graphs}
\author{Mateusz Wasilewski}
\address{Institute of Mathematics of the Polish Academy of Sciences, ul. \'{S}niadeckich 8, 00–656
Warszawa, Poland}
\email{mwasilewski@impan.pl}
\begin{abstract}
We clarify the correspondence between the two approaches to quantum graphs: via quantum adjacency matrices and via quantum relations. We show how the choice of a (possibly non-tracial) weight manifests itself on the quantum relation side and suggest an extension of the theory of quantum graphs to the infinite dimensional case. Then we use this framework to introduce quantum graphs associated to discrete quantum groups, leading to a new definition of a quantum Cayley graph.
\end{abstract}
\maketitle

\section{Introduction}
Inspired by the Lov\'{a}sz bound for the zero-error capacity of a channel, the authors of \cite{MR3015725} introduced the notion of a quantum graph on a matrix algebra $M_n$, as an object associated to a quantum channel $\Phi: M_n \to M_n$ -- its quantum confusability graph. Independently, the author of \cite{MR2908249} developed a theory of quantum relations, inspired by his previous joint work with Kuperberg \cite{MR2908248} on quantum metrics. Specialised to symmetric, reflexive relations on $M_n$, both approaches give the same answer -- operator subsystems $V \subset M_n$.

Another approach has been suggested in \cite{MR3849575}, where the authors introduced the notion of a quantum adjacency matrix. They were inspired by \cite{MR3471851}, where quantum graph homomorphisms have been introduced, viewed as perfect quantum strategies of winning the graph homomorphism game. The authors of \cite{MR3849575} developed a categorical framework in which one can talk about general (finite) quantum sets and quantum functions. A quantum graph is then a quantum set with the extra structure of a quantum adjacency matrix, and a quantum graph homomorphism is a quantum function that preserves this extra structure. This approach to quantum graphs is equivalent to the one using operator systems (see \cite{MR4402656} for an explicit correspondence). This definition has been extended in \cite{MR4091496} to include quantum graphs equipped with non-tracial states and new links have been found between quantum isomorphisms between quantum graphs and monoidal equivalence of the respective quantum automorphism groups. In \cite{2203.08716} a correspondence between the operator systems and quantum adjacency matrices has been established in the non-tracial case. In \cite{Gro22} and \cite{MR4481115} small examples, namely quantum graphs on $M_2$, have been classified.

In this work we clarify the relation between various definitions in the non-tracial case for general quantum graphs, without any symmetry assumptions. The main point here is a consistent use of the KMS inner product, as opposed to the more common GNS inner product. Using the Hilbert space techniques from \cite{2203.08716}, but with the KMS inner product replacing the GNS one, we re-establish the one-to-one correspondence between quantum adjacency matrices on a finite dimensional $C^{\ast}$-algebra $\op{B}$ and projections in $\op{B}\otimes \op{B}^{\op{op}}$. Because we do not assume GNS symmetry, no invariance under the modular group is required.

\begin{introtheorem}
There is a one-to-one correspondence between quantum adjacency matrices $A:\op{B}\to \op{B}$ and projections in $\op{B}\otimes \op{B}^{\op{op}}$. Moreover, the tensor flip of the projection corresponds to taking the KMS adjoint of $A$.
\end{introtheorem}

Another thing that has been missing is that even to define the notion of a quantum adjacency matrix we need to have a state (or a more general positive functional, depending on the normalisation convention) on our algebra, but Weaver's theory of quantum relations works for arbitrary von Neumann algebras without any reference state: a quantum relation on $\mathsf{M}\subset \op{B}(\HH)$ is a weak$^{\ast}$ closed $\mathsf{M}'$-bimodule and there is a way of making this notion representation independent. It turns out that, using an old result of Haagerup from \cite{MR549119}, a weight $\psi$ on $\mathsf{M}$ gives an operator valued weight $\psi^{-1}$ from $\op{B}(\HH)$ to $\mathsf{M}'$. In our case, i.e. if $\mathsf{M}$ is a finite dimensional $C^{\ast}$-algebra with a positive functional $\psi$ such that $mm^{\ast} = \op{Id}$ then $\psi^{-1}$ is a faithful normal conditional expectation, which endows $\op{B}(\HH)$ with a Hilbert $\mathsf{M}'$-module structure. Working with an analogue of the KMS inner product in this case we obtain a correspondence between projections in $\op{B}\otimes \op{B}^{\op{op}}$ and weak$^{\ast}$ closed $\op{B}'$-bimodules that interacts nicely with symmetry conditions.

\begin{introtheorem}
For a finite dimensional $C^{\ast}$-algebra $\op{B} \subset \op{B}(\HH)$ equipped with a positive functional $\psi$ such that $mm^{\ast}=\op{Id}$ there is a one-to-one correspondence between projections $P\in\op{B}\otimes \op{B}^{\op{op}}$ and weak$^{\ast}$ closed $\op{B}'$-bimodules $V\subset \op{B}(\HH)$. Moreover, the bimodule $V^{\ast}$ corresponds to the tensor flip of $P$.
\end{introtheorem}

We also extended the notion of a quantum graph to an infinite direct sum of matrix algebras. In this case the $\op{B}'$-valued inner product on a weak$^{\ast}$ closed bimodule is not necessarily complete, so in general we have to work with pre-Hilbert modules, but the genuine Hilbert modules form a natural class from the standpoint of quantum graph theory.
\begin{introtheorem}
Let $\op{B}:= \bigoplus_{\alpha=1}^{\infty} M_{n_{\alpha}}$ and let $A: c_{00}-\bigoplus_{\alpha=1}^{\infty} M_{n_{\alpha}} \to \op{B}$ be a quantum adjacency matrix. Then $A$ is of bounded degree, i.e. extends to a normal completely positive map $\widetilde{A}:\op{B} \to \op{B}$ if and only if for any faithful representation $\op{B} \subset \op{B}(\HH)$ the corresponding $\op{B}'$-bimodule $V$ is a (complete, self-dual) Hilbert $\op{B}'$-module.
\end{introtheorem}
Next we investigate the quantum adjacency matrices on the algebras associated to discrete quantum groups, as they are infinite direct sums of matrix algebras possessing a natural weight, namely either of the Haar weights, and they happen to satisfy the condition $mm^{\ast}=\op{Id}$, hence fit perfectly into the theory. To respect the group structure, we study quantum adjacency matrices that are covariant with respect to the right action, i.e. left convolutions.
\begin{introtheorem}
Let $\bGamma$ be a discrete quantum group and let $A:\ell^{\infty}(\bGamma)\to \ell^{\infty}(\bGamma)$ be given by $Ax:= P\ast x$ for some $P\in c_{00}(\bGamma)$, where $\ell^{\infty}(\bGamma)\simeq \ell^{\infty}-\bigoplus_{\alpha \in \op{Irr}(\mathbb{G})} M_{n_{\alpha}}$ is the algebra of bounded functions on $\bGamma$ and $c_{00}(\bGamma):= c_{00}-\bigoplus_{\alpha \in \op{Irr}(\mathbb{G})} M_{n_{\alpha}}$ is the algebra of finitely supported functions on $\bGamma$. Then $A$ is a quantum adjacency matrix if and only if $P$ is a projection. Moreover, $A$ is GNS symmetric if $P$ is invariant under the antipode of $\bGamma$ and it is KMS symmetric if $P$ is invariant under the \emph{unitary} antipode.
\end{introtheorem}
A projection $P\in c_{00}(\bGamma)$ can be viewed as a finite subset of $\bGamma$ and to define a quantum Cayley graph we need a symmetric generating subset. Symmetry should just be the KMS symmetry of the corresponding quantum adjacency matrix and the property of being generating is also not difficult to generalise to the non-commutative setting. Classically the Cayley graph is a geometric object used to study the group and its geometry depends very little on the particular generating set. Using the notion of a quantum metric from \cite{MR2908248} we arrive at the following result.
\begin{introtheorem}
Let $\bGamma$ be a discrete quantum group and let $P_1, P_2 \in c_{00}(\bGamma)$ be two generating projections. Then the corresponding quantum Cayley graphs are bi-Lipschitz equivalent.
\end{introtheorem}
We finish the introduction by quickly describing the contents of the paper. In section \ref{Sec:Prelim} we gather some information about the KMS inner product, Hilbert modules and quantum groups. In section \ref{Sec:Qgraphs} we discuss the correspondence between the three approaches to quantum graphs, and also extend it to the infinite dimensional case. Next, in section \ref{Sec:covqadj}, we finally turn to quantum groups. We describe the covariant quantum adjacency matrices and characterise the GNS/KMS symmetric ones. In the last section \ref{Sec:qcayley} the quantum Cayley graphs are introduced. We prove that changing the generating projection results in a bi-Lipschitz equivalent quantum graph. We also give an easy application: if the balls in the Cayley graph grow subexponentially then the discrete quantum group is amenable. We finish this section with some examples.
\section{Preliminaries}\label{Sec:Prelim}
We will work with pairs $(\op{B}, \psi)$, where $\op{B}$ is a finite dimensional $C^{\ast}$-algebra equipped with non-tracial positive functional $\psi$. If $\op{B}$ is just a matrix algebra $M_n$ then it means that $\psi$ is of the form $\psi(x) := \op{Tr}(\rho x)$, where $\rho$ is a positive semidefinite matrix. In this case you can define an inner product on $M_n$ by $\langle x, y \rangle := \psi(x^{\ast} y)$, which we will refer to as the GNS inner product, and if $\rho$ is invertible, then it provides a Hilbert space structure on $\op{B}$, which we will assume from now on. There is also another way to define an inner product, for which we need the notion of the modular group of a functional. For each $t\in \mathbb{R}$ we define $\sigma_{t}(x):= \rho^{i t} x \rho^{-i t}$, i.e. a conjugation by the unitary matrix $\rho^{it}$. We have $\sigma_{t} \circ \sigma_{s} = \sigma_{t+s}$ and the whole collection $(\sigma_{t})_{t\in \mathbb{R}}$ is called the modular group of $\psi$. In the finite dimensional setting we can also define the modular group at complex argument, namely for any $z\in \mathbb{C}$ we put $\sigma_{z}(x):= \rho^{iz} x \rho^{-iz}$, in particular $\sigma_{-i}(x) = \rho x \rho^{-1}$; an important and easy to check property of the modular group is that $\psi \circ \sigma_{z} = \psi$. Even though the functional $\psi$ is not tracial in general, i.e. $\psi(xy)\neq \psi(yx)$ for some $x,y \in \op{B}$, the modular group allows us to recover this property to an extent. To be more precise, it enjoys the \emph{KMS property}, i.e.
\[
\psi(x y) = \psi(y \sigma_{-i}(x)),
\]
which follows from
\[
\psi(xy) = \op{Tr}(\rho x y) = \op{Tr} (\rho x \rho^{-1} (\rho y)) = \op{Tr}(\rho y \rho x \rho^{-1}) = \op{Tr}(\rho y \sigma_{-i}(x)) = \psi (y \sigma_{-i}(x)).
\]
General finite dimensional $C^{\ast}$-algebras are direct sums of matrix algebras, so it is easy to extend all these notions.

The Hilbert space structure will be of paramount importance to us, and it is therefore crucial to specify it, especially since we will not use the typical GNS inner product induced by a positive functional. Namely, we will work with the $\emph{KMS}$ inner product, i.e. $\langle x,y \rangle_{\op{KMS}} := \psi(x^{\ast} \sigma_{-\frac{i}{2}} (y))$. One can relate it to the more common GNS inner product, yielding a more symmetric formula 
\begin{equation}\label{Eq:SymKMS}
\psi(x^{\ast} \sigma_{-\frac{i}{2}}(y)) = \psi\left((\sigma_{-\frac{i}{4}}(x))^{\ast} \sigma_{-\frac{i}{4}}(y)\right).
\end{equation}
The reason for using the KMS inner product is that it interacts nicely with positivity, namely if $\psi(x) = \op{Tr}(\rho x)$, then $\langle x,y \rangle_{\op{KMS}} = \op{Tr} ((\rho^{\frac{1}{4}} x \rho^{\frac{1}{4}})^{\ast} \rho^{\frac{1}{4}} y \rho^{\frac{1}{4}})$, i.e. we use the positivity preserving embedding $x \mapsto \rho^{\frac{1}{4}} x \rho^{\frac{1}{4}}$. The KMS inner product also interacts nicely with the adjoint, namely $\langle x, y \rangle_{\op{KMS}} = \langle y^{\ast}, x^{\ast}\rangle_{\op{KMS}}$, exactly like in the tracial case. Moreover, we will want to talk about undirected graphs, which amounts to a symmetry condition on the adjacency matrix, and GNS symmetry is too restrictive of a condition, as the following well-known statement shows.
\begin{lem}
Suppose that $A: \op{B} \to \op{B}$, where $\op{B}$ is a finite dimensional $C^{\ast}$-algebra, is GNS symmetric with respect to a faithful positive functional $\psi$, i.e. $\psi((Ax) y) = \psi(x (Ay))$. Then $A\circ \sigma_t = \sigma_t \circ A$.
\end{lem}
\begin{proof}
We have 
\begin{align*}
\psi(A(\sigma_i(x)) y) &= \psi(\sigma_{i}(x) (Ay)) = \psi((Ay) x) \\
&= \psi(y (Ax)) = \psi(\sigma_i(Ax) y),
\end{align*}
hence $A \circ \sigma_i = \sigma_i \circ A$ and this is enough to conclude that $A$ is covariant with respect to the modular group, as we will now show.

The map $A$ induces a map $\widetilde{A}(x\Omega):= A(x)\Omega$ on the level of the GNS Hilbert space $L^{2}(\op{B},\psi)$. Since by definition $\sigma_{-i}(x)\Omega = \Delta(x\Omega)$, where $\Delta: L^{2}(\op{B},\psi) \to L^{2}(\op{B},\psi)$ is the modular operator, therefore $\widetilde{A}$ commutes with $\Delta$, as $A$ commutes with $\sigma_{-i}$. It follows that $\widetilde{A} \Delta^{it} = \Delta^{it} \widetilde{A}$ for any $t\in \mathbb{R}$, hence $ A\circ \sigma_t = \sigma_t\circ A$.
\end{proof}
It is in many cases much easier to perform computations with respect to the GNS inner product, and in some instances, we can resort to that, e.g. if we want to compute the adjoint of a map that intertwines the respective modular groups.
\begin{lem}
Suppose that $A:(\op{B},\psi) \to (\op{C},\varphi)$ satisfies $A \circ \sigma_t^{\psi} = \sigma_{t}^{\varphi} \circ A$. Then the adjoint of $A$ with respect to the KMS inner products is the same as the adjoint with respect to the GNS inner products.  
\end{lem}
\begin{proof}
Follows easily from formula \eqref{Eq:SymKMS}.
\end{proof}

This lemma applies for example to the multiplication map $m: \op{B}\otimes \op{B} \to \op{B}$, because the equality $m\circ (\sigma_t \otimes \sigma_t) = \sigma_t \circ m$ just means that $\sigma_t$ is a homomorphism; we will now write down the formula for the adjoint.
\begin{prop}\label{Prop:mstar}Suppose that $M_n$ is equipped with a faithful positive functional $\psi(x):= \alpha \op{Tr} (\rho x)$, where $\alpha>0$ and $\rho \in M_n$ is a positive definite matrix. Then we have
\[
m^{\ast}(e_{ij}\rho^{-1}) = \frac{1}{\alpha}\sum_{k} e_{ik}\rho^{-1} \otimes e_{kj}\rho^{-1}
\]
and $mm^{\ast} = \frac{\op{Tr}(\rho^{-1})}{\alpha}$. Moreover $m^{\ast}$ is a bimodular map.
\end{prop}
\begin{proof}
Let us compute $\langle e_{pq} \otimes e_{rs}, m^{\ast}(e_{ij}\rho^{-1})\rangle$, using the formula for $m^{\ast}$ from the statement of the proposition:
\begin{align*}
\alpha\sum_{k} \op{Tr}(\rho e_{qp} e_{ik} \rho^{-1}) \op{Tr}(\rho e_{sr} e_{kj} \rho^{-1}) &= \alpha \sum_{k} \delta_{pi}\delta_{qk} \delta_{rk} \delta_{sj} \\
&= \alpha \delta_{pi} \delta_{qr} \delta_{sj}
\end{align*}
On the other hand, it should be equal to $\langle e_{pq} e_{rs}, e_{ij}\rho^{-1}\rangle$, i.e. $\delta_{qr}\langle e_{ps}, e_{ij}\rho^{-1}\rangle = \alpha \delta_{qr} \delta_{pi}\delta_{sj}$, which is the same.

To check that $mm^{\ast} = \frac{\op{Tr}(\rho^{-1})}{\alpha}$, simply note that $mm^{\ast}(e_{ij} \rho^{-1}) = \frac{1}{\alpha} \sum_{k} e_{ik} \rho^{-1} e_{kj} \rho^{-1}$ and $\sum_{k} e_{ik} \rho^{-1} e_{kj} = \op{Tr}(\rho^{-1}) e_{ij}$. We will now address the bimodularity of $m^{\ast}$. The case of the left action is straigthforward and for the right action we will use the KMS property. Indeed 
\begin{align*}
\langle a \otimes b, m^{\ast}(xy)\rangle &= \langle ab, xy\rangle = \psi (b^{\ast} a^{\ast} x y) \\
&= \psi (\sigma_{i}(y) b^{\ast} a^{\ast} x) = \psi ((a b \sigma_{-i}(y^{\ast}))^{\ast} x) \\
&= \langle a b \sigma_{-i}(y^{\ast}), x\rangle = \langle a \otimes b\sigma_{-i}(y^{\ast}), m^{\ast}(x)\rangle = \langle a \otimes b, m^{\ast}(x) y\rangle,  
\end{align*}
where for the last equality we need to write $m^{\ast}(x)$ as a finite sum $\sum_{i} s_{i}\otimes t_{i}$ and perform the same argument as in the second line, but in reverse.
\end{proof}
\begin{rem}
In the theory of quantum groups a concept of a \emph{$\delta$-form}, i.e. a state $\varphi$ for which $mm^{\ast} = \delta^2 \op{Id}$; this $\delta^2$ can be viewed as a sort of dimension in the non-tracial case.

Since we want to work with graphs, the natural measure to choose is the counting measure, which is not a probability measure, i.e. not a state, but it satisfies the condition $mm^{\ast}=\op{Id}$, which we will adopt also in the noncommutative case. If a positive functional $\varphi$ satisfies the condition $mm^{\ast}=\op{Id}$ then we will interpret the value $\varphi(\mathds{1})$ as the dimension of our quantum space. This choice of normalization will also make the definition of a quantum adjacency matrix cleaner.
\end{rem}
As a consequence of the previous proposition, the positive functionals satisfying $mm^{\ast}=\op{Id}$ are of the form $\op{Tr}(\rho^{-1}) \op{Tr}(\rho x)$; the extension to direct sums of matrix algebras is immediate. For further use, we record here the following simple lemma.
\begin{lem}\label{Lem:weightprojection}
Let $\psi$ be a positive functional on $M_n$ such that $mm^{\ast}=\op{Id}$ and let $P\in M_n$ be a non-zero projection. Then $\psi(P)\geqslant 1$.
\end{lem}
\begin{proof}
We have $\psi(x) = \op{Tr}(\rho^{-1}) \op{Tr}(\rho x)$. By the Cauchy-Schwarz inequality we have
\[
1 \leqslant (\op{Tr}(P))^{2} = (\op{Tr}( \rho^{-\frac{1}{2}} \rho^{\frac{1}{2}} P))^2 \leqslant \op{Tr}(\rho^{-1}) \op{Tr}(\rho P^2) = \psi(P).
\]
\end{proof}

We will occasionally have to work with the opposite algebra $\op{B}^{\op{op}}$, which is the same vector space equipped with the opposite multiplication, i.e. $x^{\op{op}} y^{\op{op}} = (yx)^{\op{op}}$, where by $x^{op}$ we mean an element $x\in \op{B}$ viewed as an element of $\op{B}^{\op{op}}$. We will denote the identity map $\op{B} \ni x \mapsto x^{\op{op}} \in \op{B}^{\op{op}}$ by $\iota$. For a positive functional $\psi$ on $\op{B}$ we can define $\psi^{\op{op}}: \op{B}^{\op{op}} \to \mathbb{C}$ by $\psi^{\op{op}}(x^{\op{op}}) := \psi(x)$. We can relate the modular group of this functional to the original one, namely $\sigma_t^{\op{op}} (x^{\op{op}})  = (\sigma_{-t}(x))^{\op{op}}$. Using the KMS property of $\psi$ and $\psi^{\op{op}}$ we obtain
\begin{align*}
\psi(x \sigma_{-i}(y)) &= \psi(yx) = \psi^{\op{op}}((yx)^{\op{op}})= \psi^{\op{op}} (x^{\op{op}} y^{\op{op}}) \\
&= \psi^{\op{op}}(\sigma_{i}^{\op{op}}(y^{\op{op}}) (x^{\op{op}})),
\end{align*}
from which it follows that $\sigma_{i}^{\op{op}}(y^{\op{op}}) = (\sigma_{-i}(y))^{\op{op}}$ and that is enough to conclude. Most often we will work with the algebra $\op{B}\otimes \op{B}^{\op{op}}$, whose natural action on $\op{B}\otimes \op{B}$ will be denoted by $\#$, i.e. $(a\otimes b^{\op{op}})\# (c\otimes d):= ac \otimes db$.

In subsection \ref{Subsec:bimod} we will use the language of Hilbert $C^{\ast}$-modules. We refer the reader to the book \cite{MR2111973} for all the necessary information. We just recall here quickly the main definition.
\begin{defn}
Let $A$ be a $C^{\ast}$-algebra and let $X$ be a right $A$-module. We call $X$ a pre-Hilbert module if it is equipped with a sesquilinear map (linear in the second variable) $\langle \cdot, \cdot \rangle: X \times X \to A$ such that:
\begin{enumerate}[{\normalfont (i)}]
\item $\langle x,ya\rangle = \langle x, y\rangle a$ for all $x,y \in X$ and $a\in A$;
\item $\langle x,y\rangle^{\ast} = \langle y, x \rangle$;
\item $\langle x, x\rangle \geqslant 0$ and $\langle x, x\rangle = 0$ if and only if $x=0$.
\end{enumerate}
We call $X$ a Hilbert module if it is complete with respect to the norm $\|x\|:= \sqrt{\|\langle x, x\rangle\|_{A}}$. A Hilbert module $X$ is self-dual if any bounded right $A$-linear map 
$\phi: X \to A$ is of the form $\phi(x) = \langle y, x\rangle$ for some $y\in X$. 
\end{defn}

In sections \ref{Sec:covqadj} and \ref{Sec:qcayley} we will work with compact and discrete quantum groups. For information about compact quantum groups we refer to the excellent book \cite{MR3204665}. For information about discrete quantum groups we refer to \cite{MR1378538}, although sometimes it is more convenient to treat them as general locally compact quantum groups and use the results from \cite{MR1832993}. Here we recall just the basic definitions.
\begin{defn}
A compact quantum group $\mathbb{G}$ is described by a pair $(A,\Delta)$, where $A$ is a unital $C^{\ast}$-algebra and $\Delta: A \to A \otimes_{\op{min}} A$ is a $\ast$-homomorphism satisfying:
\begin{enumerate}[{\normalfont{(i)}}]
\item $(\op{Id}\otimes \Delta) \circ \Delta = (\Delta\otimes \op{Id})\circ \Delta$ (coassociativity);
\item the spaces $(A\otimes \mathds{1})\Delta(A):=\op{span}\{(a\otimes \mathds{1})\Delta(b), a,b \in A\}$ and $(\mathds{1}\otimes A)\Delta(A)$ are dense in $A\otimes_{\op{min}} A$.
\end{enumerate}
Typically the algebra $A$ will be denoted by $C(\mathbb{G})$ as it is meant to generalize the algebra of continuous functions on a compact group.
\end{defn}
One can prove that the Haar measure exists for compact quantum groups, and representation theory can be developed, in particular an analogue of the Peter-Weyl theorem holds; $\op{Irr}(\mathbb{G})$ will denote the set of equivalence classes of irreducible representations of $\mathbb{G}$.
\begin{defn}
Let $\ell^{\infty}(\bGamma):= \ell^{\infty}-\bigoplus_{\alpha \in \op{Irr}(\mathbb{G})} M_{n_{\alpha}}$. There is a unique $\ast$-homomorphism $\Delta_{\bGamma}: \ell^{\infty}(\bGamma) \to \ell^{\infty}(\bGamma)\overline{\otimes} \ell^{\infty}(\bGamma)$ such that $(\Delta_{\bGamma} \otimes \op{Id})(W) = W_{23} W_{13}$, where $W \in \ell^{\infty}(\bGamma) \overline{\otimes} L^{\infty}(\mathbb{G})$ is the unitary describing the right regular representation of $\mathbb{G}$. One can also show existence of left and right Haar measures, thus $(\ell^{\infty}(\bGamma), \Delta_{\bGamma})$ becomes a locally compact quantum group, which is the discrete dual of $\mathbb{G}$. All discrete quantum groups are of this form; one can also define them from scratch by saying that the von Neumann algebra of functions is a direct sum of matrix algebras and there exists an appropriate comultiplication on this algebra. The examples that we are most interested in arise as duals of compact quantum groups, so this is the perspective that we adopted.
\end{defn}
\section{Quantum graphs}\label{Sec:Qgraphs}
Throughout this section $\op{B}$ will be a finite dimensional $C^{\ast}$-algebra equipped with a faithful positive functional $\psi$ such that $mm^{\ast} = \op{Id}$. Thus a quantum graph structure will really be an extra structure on the pair $(\op{B},\psi)$, i.e. a \emph{quantum space}, not only on the algebra $\op{B}$. However, in most situations $\psi$ will be clear from the context and will be suppressed in the notation. 
\subsection{Quantum adjacency matrices versus projections}
We will discuss a correspondence between quantum adjacency matrices $A: \op{B} \to \op{B}$ and projections in $\op{B}\otimes \op{B}^{\op{op}}$. This correspondence has already been established (see \cite[Proposition 2.3]{MR4555986} and \cite[Theorem 5.17]{2203.08716}), so we will be brief but we would like to stress how the consistent use of the KMS inner product clarifies the situation.

\begin{defn}
A quantum adjacency matrix on a quantum space $(\op{B},\psi)$ is a completely positive map $A:\op{B}\to \op{B}$ such that $m(A\otimes A)m^{\ast} = A$. 
\end{defn}
\begin{rem}\label{Rem:starvspositive}
Some authors prefer to assume that $A$ is \emph{real} (see \cite{MR4481115}), i.e. $\ast$-preserving, but it is equivalent to complete positivity under the condition $m(A\otimes A)m^{\ast} = A$. Indeed, by Proposition \ref{Prop:Generalizedchoi} the Choi matrix of such an $A$ is an idempotent and self-adjointness of an idempotent is equivalent to its positivity. 

In the classical case no extra condition is required because there all adjacency matrices are automatically completely positive or, equivalently, all idempotents are self-adjoint. From the point of view of quantum relations, our choice is like choosing the orthogonal projections among all projections onto a given subspace, and this choice ensures that all the correspondences between different definitions of quantum graphs are one-to-one.
\end{rem}
If one mimics the tracial situation, the natural object to consider seems to be $\widetilde{P}:=(A\otimes \op{Id})m^{\ast}(\mathds{1}) \in \op{B} \otimes \op{B}^{\op{op}}$, which is an idempotent. In the terminology of \cite{2203.08716}, $\widetilde{P}$ corresponds to $A$ via the map $\Psi^{\prime}_{0,0}$, which maps the rank one operator $|a\rangle\langle b|$ to $a \otimes b^{\ast} \in \op{B}\otimes \op{B}^{\op{op}}$. However, the inner product in that work is the GNS inner product and we will see now how the situation changes if we choose to work with the KMS inner product instead.
\begin{lem}\label{Lem:choirep}
Suppose that $\op{B}\simeq \bigoplus_{\alpha} M_{n_{\alpha}}$ is equipped with a weight $\psi:= \bigoplus_{\alpha} \op{Tr}(\rho_{\alpha}^{-1}) \op{Tr}(\rho_{\alpha} \cdot)$. Let $\Psi^{\op{KMS}}: \op{End}(\op{B},\langle\cdot,\cdot\rangle_{\op{KMS}}) \to \op{B} \otimes \op{B}^{\op{op}}$ be defined via $|a\rangle \langle b| \mapsto a \otimes (b^{\ast})^{\op{op}}$. Then for each $A \in \op{End}(\op{B}, \langle\cdot,\cdot\rangle_{\op{KMS}})$ we have
\begin{equation}\label{Eq:ProjfromAdj}
\Psi^{\op{KMS}}(A) = \sum_{\alpha} \frac{1}{\op{Tr}(\rho_{\alpha}^{-1})} \sum_{i,j} A\left(\rho_{\alpha}^{-\frac{1}{4}} e_{ij}^{\alpha} \rho_{\alpha}^{-\frac{1}{4}}\right) \otimes \left(\rho_{\alpha}^{-\frac{1}{4}} e_{ji}^{\alpha} \rho_{\alpha}^{-\frac{1}{4}}\right)^{\op{op}}.
\end{equation}
\end{lem} 
\begin{proof}
It suffices to check the formula \eqref{Eq:ProjfromAdj} for the rank one operators $|a\rangle\langle b|$. In this case we have ($b_{\alpha}$ will denote the component of $b$ living in the summand $M_{n_\alpha}$)
\[
\Psi^{\op{KMS}}(|a\rangle\langle b|) = \sum_{\alpha}\sum_{i,j} a \op{Tr}\left( b^{\ast}_{\alpha} \rho_{\alpha}^{\frac{1}{4}} e_{ij}^{\alpha}\rho_{\alpha}^{\frac{1}{4}} \right) \otimes \left(\rho_{\alpha}^{-\frac{1}{4}}e_{ji}^{\alpha} \rho_{\alpha}^{-\frac{1}{4}}\right)^{\op{op}}.
\]
The trace $\op{Tr}\left( b^{\ast}_{\alpha} \rho_{\alpha}^{\frac{1}{4}} e_{ij}^{\alpha}\rho_{\alpha}^{\frac{1}{4}} \right)$ just computes the $(j,i)$ entry of $\rho_{\alpha}^{\frac{1}{4}} b^{\ast}_{\alpha} \rho_{\alpha}^{\frac{1}{4}}$, hence we obtain
\[
\Psi^{\op{KMS}}(|a\rangle\langle b|) = \sum_{\alpha}\sum_{i,j} a \left(\rho_{\alpha}^{\frac{1}{4}} b^{\ast}_{\alpha} \rho_{\alpha}^{\frac{1}{4}}\right)_{ji} \otimes \left(\rho_{\alpha}^{-\frac{1}{4}}e_{ji}^{\alpha} \rho_{\alpha}^{-\frac{1}{4}}\right)^{\op{op}} = a \otimes \sum_{\alpha}\sum_{i,j} \left(\rho_{\alpha}^{-\frac{1}{4}} \left(\rho_{\alpha}^{\frac{1}{4}} b^{\ast}_{\alpha} \rho_{\alpha}^{\frac{1}{4}}\right)_{ji} e_{ji}^{\alpha} \rho_{\alpha}^{-\frac{1}{4}}\right)^{\op{op}}.
\]
For any matrix $X$ the sum $\sum_{i,j} X_{ij} e_{ij}$ is clearly equal to $X$, so the second leg reduces to $\sum_{\alpha} b^{\ast}_{\alpha} = b^{\ast}$.
\end{proof}
\begin{rem}
The elements $\frac{1}{\sqrt{\op{Tr}(\rho_{\alpha}^{-1})}} \rho_{\alpha}^{-\frac{1}{4}}e_{ij}^{\alpha} \rho_{\alpha}^{-\frac{1}{4}}$ form an orthonormal basis of $\op{B}$ with respect to the KMS inner product, so $\Psi^{\op{KMS}}(A)$ can be thought of as the matrix of $A$ with respect to this orthonormal basis.
\end{rem}
We will now relate $P:= \Psi^{\op{KMS}}(A)$ with $\widetilde{P}= (A\otimes \iota)m^{\ast}(\mathds{1})$. Before that we need a lemma (which we only write down in the case of a single matrix block, but the generalization to finite dimensional $C^{\ast}$-algebras is obvious).
\begin{lem}\label{Lem:modularidem}
Suppose $\op{B}=M_n$ is equipped with a positive functional $\psi(x):= \op{Tr}(\rho^{-1})\op{Tr}(\rho x)$, where $\rho \in M_{n}$ is positive definite. Let $\varepsilon:= m^{\ast}(\mathds{1}) = \frac{1}{\op{Tr}(\rho^{-1})} \sum_{i,j} e_{ij} \rho^{-1} \otimes e_{ji}$. Then for each $z,w \in \mathbb{C}$ we have $\varepsilon = \frac{1}{\op{Tr}(\rho^{-1})} \sum_{i,j} \rho^{z} e_{ij} \rho^{w-1} \otimes \rho^{-w} e_{ij} \rho^{-z}$. 
\end{lem}
\begin{proof}
Treat $\varepsilon$ as an element of $\op{B}\otimes \op{B}$. Note first that $(\sigma_{z})^{\ast} = \sigma_{-\overline{z}}$ because
\[
\langle x, \sigma_{z} y\rangle = \psi( x^{\ast} \sigma_{z}(y)) = \psi (\sigma_{-z}(x^{\ast}) y) = \psi((\sigma_{-\overline{z}}(x))^{\ast} y) = \langle \sigma_{-\overline{z}}(x), y\rangle.
\]
By taking the adjoint of the equation $m\circ (\sigma_{z}\otimes \sigma_{z}) = \sigma_{z}\circ m$ (for all $z\in\mathbb{C}$), we get that $m^{\ast}$ commutes with the modular group, i.e. $(\sigma_{z}\otimes \sigma_{z})\circ m^{\ast} = m^{\ast} \circ \sigma_{z}$ for all $z\in \mathbb{C}$. From this we get $\varepsilon = (\sigma_{-ix} \otimes \sigma_{-ix})\varepsilon$ for all $x\in \mathbb{C}$. Moreover, $m^{\ast}$ is a bimodular map, therefore $m^{\ast}(\mathds{1}) = m^{\ast}(\rho^{y} \mathds{1} \rho^{-y}) = \rho^{y} \varepsilon \rho^{-y}$ for all $y\in \mathbb{C}$. Combining the two, we obtain
\[
\varepsilon = \frac{1}{\op{Tr}(\rho^{-1})} \sum_{i,j} \rho^{x+y} e_{ij} \rho^{-1-x} \otimes \rho^{x} e_{ji} \rho^{-y-x}.
\]
Now put $z=x+y$ and $w=-x$. 
\end{proof}
\begin{lem}\label{Lem:Representingproj}
Let $A:\op{B} \to \op{B}$. We have $\Psi^{\op{KMS}}(A) = (A\otimes (\iota\circ \sigma_{-\frac{i}{2}}))m^{\ast}(\mathds{1})$, where $\iota: \op{B}\to \op{B}^{\op{op}}$ is the identity map $\iota(x):= x^{\op{op}}$.

Since $\iota \circ \sigma_{-\frac{i}{2}} = \sigma^{\op{op}}_{\frac{i}{2}}\circ \iota$, we get $P = (\op{Id} \otimes \sigma^{\op{op}}_{\frac{i}{2}})(\widetilde{P})$.
\end{lem}
\begin{proof}
From the previous lemma we get that 
\[
m^{\ast}(\mathds{1}) = \sum_{\alpha} \frac{1}{\op{Tr}(\rho_{\alpha}^{-1})}\sum_{i,j} \rho_{\alpha}^{-\frac{1}{4}} e_{ij}^{\alpha} \rho_{\alpha}^{-\frac{1}{4}} \otimes \rho_{\alpha}^{-\frac{3}{4}} e_{ji}^{\alpha} \rho_{\alpha}^{\frac{1}{4}}.
\]
Note that 
\[
\rho_{\alpha}^{-\frac{3}{4}} e_{ji}^{\alpha} \rho_{\alpha}^{\frac{1}{4}} = \rho_{\alpha}^{-\frac{1}{2}} (\rho_{\alpha}^{-\frac{1}{4}} e_{ji}^{\alpha} \rho_{\alpha}^{-\frac{1}{4}}) \rho_{\alpha}^{\frac{1}{2}} = \sigma_{\frac{i}{2}}(\rho_{\alpha}^{-\frac{1}{4}} e_{ji}^{\alpha} \rho_{\alpha}^{-\frac{1}{4}}),
\]
which ends the proof.
\end{proof}

We are now ready to summarize the properties of $\Psi^{\op{KMS}}$.
\begin{prop}\label{Prop:Generalizedchoi}
The map $\Psi^{\op{KMS}}: \op{End}(\op{B},\langle\cdot,\cdot\rangle_{\op{KMS}}) \to \op{B}\otimes \op{B}^{\op{op}}$ establishes a one-to-one correspondence between linear maps on $\op{B}$ and elements of $\op{B}\otimes \op{B}^{\op{op}}$. We have:
\begin{enumerate}[{\normalfont (i)}]
\item $\Psi^{\op{KMS}}(A)$ is an idempotent if and only if $A$ satisfies $m(A\otimes A)m^{\ast} = A$;
\item $\Psi^{\op{KMS}}(A)$ is self-adjoint if and only if $A$ is $\ast$-preserving;
\item $\Psi^{\op{KMS}}(A)$ is positive if and only if $A$ is completely positive;
\item $\Psi^{\op{KMS}}(A)$ is invariant under the flip map $\Sigma: \op{B}\otimes \op{B}^{\op{op}} \to \op{B}\otimes \op{B}^{\op{op}}$ if and only if the KMS adjoint $A^{\ast}_{\op{KMS}}$ is equal to the map $\overline{A}(x):= (A(x^{\ast}))^{\ast}$.
\end{enumerate}
In particular, if $A: \op{B}\to \op{B}$ is a quantum adjacency matrix then it is self-adjoint with respect to the KMS inner product if and only if $\Sigma(\Psi^{\op{KMS}}(A)) = \Psi^{\op{KMS}}(A)$.

\end{prop}
\begin{proof}
The first bullet point follows from \cite[Proposition 2.3(2)]{MR4555986}, which states that $(A\otimes \iota)m^{\ast}(\mathds{1})$ is an idempotent iff $m(A\otimes A)m^{\ast} = A$ and Lemma \ref{Lem:Representingproj}, from which we infer that $\Psi^{\op{KMS}}(A)$ and $(A\otimes \iota)m^{\ast}(\mathds{1})$ are related by a multiplicative map (the modular group). The only new statement is the last one. Note that $\Psi^{\op{KMS}}(|a\rangle\langle b|) = a\otimes (b^{\ast})^{\op{op}}$, hence the flip is equal to $b^{\ast} \otimes a^{\op{op}}$, which is equal to $\Psi^{\op{KMS}}(|b^{\ast}\rangle\langle a^{\ast}|)$. Let us compute $\overline{|a\rangle\langle b|}(c) = (a \langle b, c^{\ast}\rangle_{\op{KMS}})^{\ast} = a^{\ast} \langle c^{\ast}, b\rangle_{\op{KMS}} = a^{\ast} \langle b^{\ast} , c \rangle_{\op{KMS}}$, using the properties of the KMS inner product. Using the fact that $\left(| a \rangle \langle b|\right)^{\ast} = |b\rangle \langle a|$, we conclude that $\Sigma (\Psi^{\op{KMS}}(A)) = \Psi^{\op{KMS}}(\overline{A}^{\ast}_{\op{KMS}})$, where $\overline{A}^{\ast}_{\op{KMS}}$ should be viewed as the transpose of $A$.
\end{proof}

We will end this subsection with some comments about recovering $A$ from $P:=\Psi^{\op{KMS}}(A)$. By \cite[Proposition 2.3 (1)]{MR4555986} we have $Ax = (\op{Id} \otimes \psi)(P_1(1\otimes x))$, where $P_1 := (A\otimes \op{Id})m^{\ast}(\mathds{1})$ is treated is an element of $\op{B} \otimes \op{B}$ such that $\widetilde{P} = (\op{Id} \otimes \iota)P_1$, thus we can rewrite this formula as $Ax = (\op{Id}\otimes \varphi)(1\otimes x)\# \widetilde{P})$. Recall from Lemma \ref{Lem:Representingproj} that $\widetilde{P} = (\op{Id} \otimes \sigma^{\op{op}}_{-\frac{i}{2}})(P)$. This is to be expected, because if we have a positive element $x \in \op{C}$, where $\op{C}$ is a von Neumann algebra equipped with a normal faithful positive functional $\varphi$, and we want to use it as a density of a positive functional, then we have to choose $\varphi (\cdot \sigma^{\varphi}_{-\frac{i}{2}}(x))$, not $\varphi(\cdot x)$. There is an analogous formula involving the left multiplication.

\subsection{Projections versus bimodules}\label{Subsec:bimod}

Given a representation of $\op{B}$ on a Hilbert space $\HH$ (possibly infinite dimensional), we will now establish a one-to-one correspondence between weak$^{\ast}$ closed $\op{B}'$-bimodules inside $\op{B}(\HH)$ and projections in $\op{B}\otimes \op{B}^{\op{op}}$. Before that we need a quick discussion of operator valued weights. We will not go into the details, let us just say that operator valued weights are to conditional expectations what weights are to states, i.e. they satisfy a positivity and bimodularity conditions but are in general unbounded. In this work we will only see weights and bona fide conditional expectations, that is why we decided to avoid the unnecessary technical details.

In \cite[Theorem 6.13]{MR549119} Haagerup established for a pair of von Neumann algebras $\mathsf{N} \subset \mathsf{M}\subset \op{B}(\HH)$ a bijection between normal, faithful, semifinite (abbreviated henceforth nfs) operator valued weights from $\mathsf{M}$ to $\mathsf{N}$ and nfs operator valued weights from $\mathsf{N}'$ to $\mathsf{M}'$. In particular, any weight $\psi$ on a von Neumann algebra $\mathsf{M}$ gives rise to an operator valued weight $\psi^{-1}$ from $\op{B}(\HH)$ to $\mathsf{M}'$. In \cite[Corollary 16]{MR561983} Connes obtained a formula for this operator valued weight as $\psi^{-1} (|\xi\rangle \langle \xi|) = R_{\psi}(\xi) R_{\psi}(\xi)^{\ast}$, where $R_{\psi}(\xi): \HH_{\psi} \to \HH$ is the map from the GNS space of $\psi$ given by $\Lambda(a) \mapsto a\xi$, assuming that it is bounded, that is to say that $\xi$ is a $\psi$-bounded vector (such vectors are always dense in $\HH$). 

We will now check that if $\op{B}$ is a finite dimensional $C^{\ast}$-algebra equipped with a positive functional $\psi$ such that $mm^{\ast}=\op{Id}$ then $\psi^{-1}$ is a faithful conditional expectation. First, note that it is sufficient to check for a single matrix algebra. Indeed, if $z \in \op{B}$ is a central projection (namely $z \in \op{B} \cap \op{B}'$) then (using bimodularity of $\psi^{-1}$ over $\op{B}'$) $\psi^{-1} (zx) = z \psi^{-1}(x) = z \psi^{-1}(x) z = \psi^{-1}(zxz)$. It follows that if $z_{\alpha}$'s are minimal central projections in $\op{B}$ and $\psi_{\alpha}$ is the restriction of the functional $\psi$ to $z_{\alpha}\op{B} \subset \op{B}(z_{\alpha} \HH)$ then $\psi^{-1}$ can be built as a direct sum of the operator valued weights $\psi_{\alpha}^{-1}$. 

We assume now that $\op{B} = M_n \subset \op{B}(\HH)$ ($\HH$ can be infinite dimensional) is equipped with a positive functional $\psi(x):=\op{Tr}(\rho^{-1}) \op{Tr}(\rho x)$, where $\rho \in M_n$ is positive definite. It is well-known that $\HH \simeq \C^{n} \otimes \K$, where $M_n$ acts as $A(v\otimes \xi):= Av \otimes \xi$. Any tensor $v\otimes \xi$ is $\psi$-bounded and one can check that $R_{\psi}(v\otimes \xi)^{\ast}: \C^{n} \otimes \K \to H_{\psi}$ is given by $R_{\psi}(v\otimes \xi)^{\ast}(w\otimes \eta) = \frac{1}{\op{Tr}(\rho^{-1})} |w\rangle\langle \rho^{-1} v|  \langle \xi, \eta\rangle$. It follows that $R_{\psi}(v\otimes \xi) R_{\psi}(v\otimes \xi)^{\ast}(w\otimes \eta) = \frac{1}{\op{Tr}(\rho^{-1})} \langle \rho^{-1} v, v\rangle \langle \xi, \eta\rangle w\otimes \xi$. In other words, if we view $\op{B}(\C^n \otimes \K)$ as $M_n \overline{\otimes} \op{B}(\K)$ and $|v \otimes \xi\rangle \langle v\otimes \xi|$ as $|v\rangle\langle v| \otimes |\xi\rangle \langle \xi|$ then $\psi^{-1} (A\otimes T) = \frac{1}{\op{Tr}(\rho^{-1})} \op{Tr}(\rho^{-1} A) \otimes T$, which amounts to slicing the first leg with the \emph{state} $\frac{1}{\op{Tr}(\rho^{-1})} \op{Tr}(\rho^{-1} \cdot)$ (see also \cite[Remark 3.3 (iii)]{MR1662525}). This is clearly a conditional expectation. Let us summarize this discussion.
\begin{prop}\label{Prop:Inversecondexp}
Let $\op{B} \subset \op{B}(\HH)$ be a finite dimensional $C^{\ast}$-algebra equipped with a positive functional $\psi$ such that $mm^{\ast}= \op{Id}$. Then $\psi^{-1}: \op{B}(\HH) \to \op{B}'$ is a normal faithful conditional expectation.
\end{prop}
\begin{rem}
Note that the statement also holds for an infinite direct sum of matrix algebras, which will be useful in the sequel (the proof remains the same). We cannot, however, hope for a much greater generality, as existence of a normal conditional expectation from $\op{B}(\HH)$ onto $\op{B}'$ implies that $\op{B}'$ is atomic \cite[Exercise IX.4.1(e)]{MR1943006}.
\end{rem}
Using this conditional expectation we can equip $\op{B}(\HH)$ with a structure of a pre-Hilbert $C^{\ast}$-module over $\op{B}'$, taking the inner product as $\langle S,T\rangle:= \psi^{-1}(S^{\ast}T)$. Following \cite{MR2908249} we will define quantum graphs on $\op{B}$ as quantum relations (dropping for now the symmetry and reflexivity conditions).

\begin{defn}
Let $\op{B} \subset \op{B}(\HH)$ be a finite dimensional $C^{\ast}$-algebra equipped with a positive functional $\psi$ such that $mm^{\ast} = \op{Id}$. A quantum graph on $(\op{B},\psi)$ is a weak$^{\ast}$ closed $\op{B}'$-bimodule $V \subset \op{B}(\HH)$. Using $\psi^{-1}$ to define a $\op{B}'$-valued inner product, we can make $V$ into a pre-Hilbert $\op{B}'$-module. 
\end{defn}

\begin{prop}\label{Prop:selfdual}
Let $V \subset \op{B}(\HH)$ be a quantum graph on $(\op{B}, \psi)$. Then $V$ is a self-dual Hilbert module, i.e. it is complete for the Hilbert module norm, and any right $\op{B}'$-linear map $T: V \to \op{B}'$ is of the form $Tv= \langle w, v\rangle$ for some $w\in V$.
\end{prop}
\begin{proof}
We will start by checking that $\op{B}(\HH)$ is a Hilbert module. Indeed, the inclusion $\op{B}' \subset \op{B}(\HH)$ is of finite index in the sense of \cite[Definition 3.4]{MR1662525}, because the index is equal to $(\psi^{-1})^{-1}(\mathds{1}) = \psi(\mathds{1}) < \infty$, but also in the sense of \cite[Proposition 3.3]{MR945550}, as from the explicit formula for $\psi^{-1}$ and the fact that we only have finitely many summands it immediately follows that for some $K>0$ we have $x\leqslant K \psi^{-1}(x)$ for all positive $x\in \op{B}(\HH)$. As a consequence, the operator norm is equivalent to the pre-Hilbert norm induced by $\psi^{-1}$ (because the other inequality $\|\psi^{-1}(x)\|\leqslant \|x\|$ is clear), hence the latter is complete. Clearly each closed submodule, such as $V$, is then also a Hilbert $\op{B}'$-module.

To prove self-duality, note that $V$ is a dual Banach space, being a weak$^{\ast}$ closed subspace of $\op{B}(\HH)$; moreover, the $\op{B}'$-valued inner product is separately weak$^{\ast}$ continuous. It follows from \cite[Lemma 8.5.4]{MR2111973} that $V$ is self-dual. It  implies that it admits an orthonormal basis, all bounded right $\op{B}'$-modular maps on $V$ are automatically adjointable and the algebra of adjointable operators is a von Neumann algebra \cite[Theorem 3.12, Proposition 3.4, and Proposition 3.10]{MR355613}).
\end{proof}
\begin{rem}
In all other examples appearing in this paper the existence of a predual making the inner product separately weak$^{\ast}$ continuous will always be easy, so checking self-duality will amount to proving completeness of the pre-Hilbert norm.
\end{rem}

We have all the necessary ingredients to associate a projection in $\op{B}\otimes \op{B}^{\op{op}}$ to a $\op{B}'$-bimodule $V$; indeed, we can take an orthonormal basis of $V$ and use it to construct a $\op{B}'$-bimodular orthogonal projection onto $V$, which in turn we can represent using $\op{B} \otimes \op{B}^{\op{op}}$. There is, however, a slight problem, since our $\op{B}'$-valued inner product is rather an analogue of the GNS inner product, not the KMS inner product, so it would not interact nicely with an additional symmetry condition; we need therefore a modular group of the operator valued weight $\psi^{-1}$. Haagerup proved that for any normal, faithful, semifinite weight $\psi'$ on $\op{B}'$ the restriction of the modular group of $\phi:=\psi' \circ \psi^{-1}$ to $\op{B}$ does not depend on $\psi'$ and moreover $\sigma^{\phi}_t = \sigma^{\psi}_{-t}$ (see \cite[Proposition 6.1 and Theorem 6.13]{MR549119}). Once we get to the explicit correspondence between various definitions of quantum graphs, we will write everything in terms of the modular group of $\psi$. However, for now we will denote the modular group of the conditional expectation $\psi^{-1}$ by $\sigma_{t}^{\psi^{-1}}$. 

To an extent, one can define this modular group just by using the Hilbert module structure of $\op{B}(\HH)$; the right multiplication operator by an element $b \in \op{B}$ is adjointable, and as it is a left $\op{B}(\HH)$-modular map, it is also a right multiplication by some element $\widetilde{b}$, which has to be equal to $\sigma_{-i}(b^{\ast})$ by the KMS property. 

This uniqueness also implies that we have a well defined action on bimodules. Indeed, if we have two weights $\phi_1:= \psi'_{1}\circ \psi^{-1}$ and $\phi_2:= \psi'_{2}\circ \psi^{-1}$ on $\op{B}(\HH)$ then their respective modular groups are given using density matrices: $\sigma^{\phi_1}_{t}(x) = \rho_1^{it} x \rho_1^{-it}$ and $\sigma^{\phi_2}_{t}(x) = \rho_2^{it} x \rho_2^{-it}$. Since they agree on $B$, we have $\rho_1^{it} \rho_2^{-it} \in \op{B}'$, hence $\rho_{1}^{it} = u_{t} \rho_{2}^{it}$ for some unitary in $\op{B}'$. As $V$ is a $\op{B}'$-bimodule, we have $u_{t} V u_{t}^{\ast} = V$, therefore $\sigma_t^{\phi_1}(V) = \sigma_t^{\phi_2}(V)$. 

Suppose now that we have a finite orthonormal set $X:=(x_1,\dots, x_d)$ in a von Neumann algebra $\mathsf{M}$ (equipped with a state $\varphi$) with respect to the KMS inner product, which consists of analytic elements. We would like to  find a relation between orthogonal projections with respect to KMS and GNS inner products. The orthogonal projection onto $V$, the span of $X$, is given by $P^{\op{KMS}}_{V}(x):= \sum_{i=1}^{d} x_i \langle x_i, x \rangle_{\op{KMS}} = \sum_{i=1}^{d} x_{i} \varphi \left(\left(\sigma_{-\frac{i}{4}}(x_{i})\right)^{\ast} \sigma_{-\frac{i}{4}}(x)\right)$. This is equal to $\sum_{i=1}^{d} x_{i} \langle \sigma_{-\frac{i}{4}}(x_i), \sigma_{-\frac{i}{4}} (x) \rangle$. Thus $\sigma_{-\frac{i}{4}}(X)$ is an orthonormal set with respect to the GNS inner product and we checked that $P^{\op{KMS}}_V = \sigma_{\frac{i}{4}}\circ P^{GNS}_{\sigma_{-\frac{i}{4}}(V)} \circ \sigma_{-\frac{i}{4}}$.
Therefore, if in our setting of bimodules we can define the bimodule $\sigma_{-\frac{i}{4}}^{\psi^{-1}}(V)$ and represent the orthogonal projection onto it as an element of $\op{B} \otimes \op{B}^{\op{op}}$ then we can also apply the conjugation by $\sigma_{\frac{i}{4}}^{\psi^{-1}}$ on the level of $\op{B}\otimes \op{B}^{\op{op}}$, where the modular group is well defined, and in this way we will obtain a representation of the orthogonal projection onto $V$ with respect to the KMS inner product.

To check that one can define $\sigma_{-\frac{i}{4}}^{\psi^{-1}}(V)$, i.e. we can analytically extend the action of $\sigma_{t}^{\psi^{-1}}$, recall the discussion preceding Proposition \ref{Prop:Inversecondexp}. We have $\op{B} = \bigoplus_{\alpha} M_{n_{\alpha}}$ and any representation is essentially a direct sum of inflations of the standard representations of matrix algebras, i.e. $\HH\simeq \bigoplus_{\alpha} \C^{n_{\alpha}} \otimes \K_{\alpha}$. The commutant $\op{B}'$ is then equal to $\bigoplus_{\alpha} \mathds{1}_{\alpha} \otimes \op{B}(\K_{\alpha})$. Since the action of the modular group $\sigma_t^{\psi^{-1}}(V)$ does not depend on the choice of the weight on $\op{B}'$, we might as well assume that the weight on $\op{B}(\HH)$ is given by the density $\bigoplus_{\alpha} \rho_{\alpha}^{-1} \otimes \op{Id}_{\K_{\alpha}}$, i.e. we take the trace on each $\op{B}(\K_{\alpha})$. Then something nontrivial happens only on  the finite dimensional subspace $\op{B}(\C^{n_{\alpha}}, \C^{n_{\beta}})$, thus we may define $\sigma_{z}(V)$ for arbitrary $z \in \mathbb{C}$.
\begin{lem}\label{Lem:Analyticbimod}
For any weak$^{\ast}$ closed $\op{B}'$-bimodule $V \subset \op{B}(\HH)$ and $z \in \mathbb{C}$ we can define another weak$^{\ast}$ closed $\op{B}'$-bimodule $\sigma_{z}^{\psi^{-1}}(V)$.
\end{lem}
By \cite[Corollary 8.5.8]{MR2111973} any bounded module map between self-dual Hilbert modules over von Neumann algebras is automatically weak$^{\ast}$ continuous, hence the orthogonal projection $P_{\sigma_{-\frac{i}{4}}^{\psi^{-1}}(V)}^{GNS}: \op{B}(\HH)\to \sigma_{-\frac{i}{4}}^{\psi^{-1}}(V)\subset \op{B}(\HH)$ is a normal $\op{B}'$-bimodular map, therefore of the form $P_{\sigma_{-\frac{i}{4}}^{\psi^{-1}}(V)}^{GNS}x = \sum_{k} b_{k} x c_{k}$ for some $b_{k}, c_{k} \in \op{B}$ (see \cite[Theorem 3.1]{MR1138841} for a slightly more general result; there the von Neumann algebraic result we need is attributed to Haagerup). In this case the orthogonal projection onto $V$ with respect to the KMS inner product would be given by $P_{V}^{\op{KMS}}(x) = \sum_{k} \sigma_{\frac{i}{4}}^{\psi^{-1}}(b_k) x \sigma_{\frac{i}{4}}^{\psi^{-1}}(c_k)$. Note that the map $x \mapsto \sum_{k} b_k x c_k$ is represented by $\sum_{k} b_{k} \otimes \left(\sigma_{\frac{i}{2}}^{\psi^{-1}}(c_k)\right)^{\op{op}}$, because the right action needs to be twisted by the modular group so that it becomes a $\ast$-representation. To wit, an element $b\otimes c^{\op{op}}\in \op{B}\otimes \op{B}^{\op{op}}$ acts as $x\mapsto bx \sigma_{-\frac{i}{2}}^{\psi^{-1}}(c)$, when we use the GNS inner product. However, if we use the KMS inner product, then $b\otimes c^{\op{op}}$ needs to acts as $x \mapsto \sigma_{\frac{i}{4}}^{\psi^{-1}}(b)x \sigma_{-\frac{i}{4}}^{\psi^{-1}}(c)$. Therefore $\sum_{k} b_{k} \otimes \left(\sigma_{\frac{i}{2}}^{\psi^{-1}}(c_k)\right)^{\op{op}}$ acting on $\op{B}(\HH)$ with respect the KMS inner product acts as $x \mapsto \sum_{k} \sigma_{\frac{i}{4}}^{\psi^{-1}}(b_k) x \sigma_{\frac{i}{4}}^{\psi^{-1}}(c_k)$, i.e. it is equal to $P_{V}^{\op{KMS}}$. It means that $P_{V}^{\op{KMS}}$ and $P_{\sigma_{-\frac{i}{4}}^{\psi^{-1}}(V)}^{GNS}$ are represented by the same element of $\op{B}\otimes \op{B}^{\op{op}}$.
\begin{prop}\label{Prop:Bimodproj}
The assignment $\Phi: \op{B}\otimes \op{B}^{\op{op}} \to \phantom{B}_{\op{B}'}\op{B}_{\op{B}'}(\HH)$, where $\phantom{B}_{\op{B}'}\op{B}_{\op{B}'}(\HH)$ denotes the algebra of bounded $\op{B}'$-bimodular maps, given by the extension of $b\otimes c^{\op{op}} \mapsto (x\mapsto b x \sigma_{-\frac{i}{2}}(c))$, is an isomorphism of von Neumann algebras.
\end{prop}
\begin{proof}
$\Phi$ is a $\ast$-homomorphism, which is also surjective by the aforementioned representation of $\op{B}'$-bimodular maps. Let us check that it is injective. Suppose that (remember that $\op{Id}\otimes \sigma_{\frac{i}{2}}^{\psi^{-1}}: \op{B}\otimes \op{B}^{\op{op}} \to\op{B}\otimes \op{B}^{\op{op}}$ is a bijection) $\Phi(\sum_{k} b_{k} \otimes \sigma_{\frac{i}{2}}(c_{k})^{\op{op}})=0$, i.e. $\sum_{k} b_{k} x c_k = 0$ for each $x\in \op{B}(\HH)$. This holds in particular for any Hilbert-Schmidt operator, identified with a vector $\xi \in \HH\otimes \overline{\HH}$. The natural representation of $\op{B}\otimes \op{B}^{\op{op}}$ on $\HH\otimes \overline{\HH}$ is clearly faithful, so injectivity follows.
\end{proof}

\begin{thm}\label{Thm:projbimod}
Let $\op{B}\subset \op{B}(\HH)$ be a finite dimensional $C^{\ast}$-algebra equipped with a positive functional $\psi$ such that $mm^{\ast} = \op{Id}$ and let $\psi^{-1}:\op{B}(\HH) \to \op{B}'$ be the corresponding normal, faithful conditional expectation. There is a one-to-one correspondence between weak$^{\ast}$ closed $\op{B}'$-bimodules $V \subset \op{B}(\HH)$ and projections $P \in \op{B}\otimes \op{B}^{\op{op}}$ obtained in the following way. For $V\subset \op{B}(\HH)$ let $P_{\sigma_{-\frac{i}{4}}^{\psi^{-1}}(V)}$ be the  $\op{B}'$-bimodular projection onto $\sigma_{-\frac{i}{4}}^{\psi^{-1}}(V)$ and let $P \in \op{B}\otimes \op{B}^{\op{op}}$ be its representing projection. On the other hand for a projection $P \in \op{B}\otimes \op{B}^{\op{op}}$ let $\widetilde{P}$ be the corresponding $\op{B}'$-bimodular projection; then $V = \sigma_{\frac{i}{4}}^{\psi^{-1}}(\op{Im}(\widetilde{P}))$.
\end{thm}

\begin{proof}
The correspondence between projections in $\op{B}\otimes \op{B}^{\op{op}}$ and $\op{B}'$-bimodular projections on $\op{B}(\HH)$ was established in Propostion \ref{Prop:Bimodproj}. The fact that for any weak$^{\ast}$ closed bimodule we can define $\sigma_{z}(V)$ for any $z\in\mathbb{C}$ was covered in Lemma \ref{Lem:Analyticbimod}. The last ingredient is the correspondence between bimodules and bimodular projections, which follows from self-duality (see Proposition \ref{Prop:selfdual}).
\end{proof}
\begin{rem}
In Subsection \ref{Subsec:Explicit} we will present a much less abstract version of this correspondence.
\end{rem}

\subsection{Infinite quantum graphs}
In order to work with infinite discrete quantum groups, we need to deal with quantum graphs defined on an infinite direct sum of matrix algebras. The theory of quantum relations developed by Weaver works equally well for general von Neumann algebras and if $\mathsf{M} \subset \op{B}(\HH)$ then a quantum relation over $\mathsf{M}$ is a weak$^{\ast}$ closed $\mathsf{M}'$-bimodule inside $\op{B}(\HH)$. We will assume that our algebra is of the form $\op{B} = \ell^{\infty}-\bigoplus_{\alpha} \op{M}_{n_{\alpha}}$ and it is equipped with a weight $\psi(x):= \bigoplus_{\alpha} \op{Tr}(\rho_{\alpha}^{-1}) \op{Tr}(\rho_{\alpha} x_{\alpha})$; such a weight satisfies $mm^{\ast}=\op{Id}$, e.g. if we view it as a map on $L^{2}(\op{B})$, but it should be noted that $m^{\ast}$ is not a well-defined map from $\op{B}$ to $\op{B}\overline{\otimes} \op{B}$. In this case Weaver's quantum relations correspond to projections in $\op{B} \overline{\otimes} \op{B}^{\op{op}}$, which in turn will be in bijection with suitably defined quantum adjacency matrices. Later in this section we will describe a certain subclass of quantum graphs that is more suitable for describing quantum Cayley graphs.
\begin{defn}
Let $I$ be a countable index set and let $\op{B}:=\ell^{\infty}-\bigoplus_{\alpha \in I} M_{n_{\alpha}}$, where $n_{\alpha} \in \mathbb{N}$. Let $c_{00}(\op{B})$ be the algebraic direct sum $c_{00}-\bigoplus_{\alpha} M_{n_{\alpha}}$, which can be more abstractly defined as the subalgebra generated by $\psi$-finite projections, i.e. projections $p$ such that $\psi(p)<\infty$.
A completely positive map $A: c_{00}(\op{B}) \to \op{B}$ is called a \textbf{quantum adjacency matrix} if for each $\alpha \in I$ it satisfies
\[
\frac{1}{\op{Tr}(\rho_{\alpha}^{-1})} \sum_{l} A(e_{il}^{\alpha}\rho_{\alpha}^{-1})A(e_{lj}^{\alpha}) = A(e_{ij}^{\alpha}).
\]
The corresponding projection in $\op{B}\overline{\otimes} \op{B}^{\op{op}}$ is given by $P:=\sum_{\alpha} \frac{1}{\op{Tr}(\rho_{\alpha}^{-1})} \sum_{i,j} A(\rho_{\alpha}^{-\frac{1}{4}}e_{ij}^{\alpha}\rho_{\alpha}^{-\frac{1}{4}}) \otimes \rho_{\alpha}^{-\frac{1}{4}}e_{ji}^{\alpha}\rho_{\alpha}^{-\frac{1}{4}}$.
\end{defn}
\begin{rem}
Note that in the definition of $P$ only the sum over $\alpha$ is infinite but the summands are projections lying in disjoint blocks, so it clearly converges in the weak/strong topology.
\end{rem}
\begin{prop}
Let $P\in \op{B} \overline{\otimes} \op{B}^{\op{op}}$. Then we can define a quantum adjacency matrix $A: c_{00}(\op{B}) \to \op{B}$ via
\[
A(e_{ij}^{\alpha}):= (\op{Id}\otimes \psi)(P\# (1\otimes \rho_{\alpha}^{-\frac{1}{2}}e_{ij}^{\alpha} \rho_{\alpha}^{\frac{1}{2}})).
\]
\end{prop}
\begin{proof}
Note that $P\# (1\otimes \rho_{\alpha}^{-\frac{1}{2}}e_{ij}^{\alpha} \rho_{\alpha}^{\frac{1}{2}}) \in \op{B}\otimes M_{n_{\alpha}}$, so slicing with $\psi$ is certainly possible. Because there are no analytical difficulties, one can check that $A$ is a quantum adjacency matrix exactly the same as in the finite dimensional case. 
\end{proof}

The next result has been proved by Weaver in the finite dimensional case (\cite[Proposition 2.23]{MR2908249}).
\begin{prop}
There is a one-to-one correspondence between projections $P\in \op{B}\overline{\otimes} \op{B}^{\op{op}}$ and weak$^{\ast}$ closed $\op{B}'$-bimodules $V\subset \op{B}(\HH)$.
\end{prop}

\begin{proof}
Let $(\mathds{1}_{\alpha})_{\alpha}$ be the minimal central projections of $\op{B}$. Then $V_{\beta\alpha}:=\mathbf{1}_{\beta} V \mathbf{1}_{\alpha}$ is a $\op{B}'$-bimodule, but the only nontrivial action comes from $M_{n_{\beta}} \otimes M_{n_{\alpha}}^{\op{op}}$, hence finding the projection onto $V_{\beta\alpha}$ is done exactly the same as in the finite dimensional case. Moreover $V$ consists precisely of those bounded operators $T\in \op{B}(\HH)$ such that $\mathbf{1}_{\beta} T \mathbf{1}_{\alpha} \in V_{\beta\alpha}$, which is a consequence of the condition that $V$ is weak$^{\ast}$ closed. Indeed, if $\mathbf{1}_{\beta} T \mathbf{1}_{\alpha} \in V_{\beta\alpha}$ then $T_n:= (\sum_{\alpha=1}^{n} \mathbf{1}_{\alpha})T (\sum_{\alpha=1}^{n}\mathbf{1}_{\alpha}) \in V$ and $T_n$ converges to $T$ in the weak$^{\ast}$ topology. To summarize, any weak$^{\ast}$ closed $\op{B}'$-bimodule can be equivalently described by the collection $(V_{\beta\alpha})_{\beta, \alpha}$. Each $V_{\beta\alpha}$ corresponds to a projection $P_{\beta\alpha} \in M_{n_{\beta}}\otimes M_{n_{\alpha}}^{\op{op}}$, and such a collection $(P_{\beta\alpha})_{\beta,\alpha}$ is exactly a projection $P\in \op{B}\overline{\otimes} \op{B}^{\op{op}}$.
\end{proof}

The established correspondence is very satisfactory because it includes all quantum relations on $\op{B}$ but the quantum adjacency matrix has poor analytic and algebraic properties, e.g. it cannot be iterated because the codomain does not match the domain. There is a natural subclass of quantum graphs that is better behaved and should include all quantum Cayley graphs. Classically, one constructs Cayley graphs from finite generating sets, hence the resulting graphs have bounded degree; this is the class that we would like to generalize.
\begin{lem}
Let $A=(A_{ij})_{i,j\in \mathbb{N}}$ be a $\{0,1\}$ matrix. It defines a normal completely positive map $A: \ell^{\infty} \to \ell^{\infty}$ if and only if it is of bounded degree, i.e. $ \sup_{i} |\{j\in \mathbb{N}: a_{ij}=1\}| < \infty$.
\end{lem}
\begin{proof}
If $A$ defines a map from $\ell^{\infty}$ to itself then we look at $A\mathds{1}$. It is clear that $(A\mathds{1})_{i} = \sum_{j} a_{ij} = |\{j: a_{ij}=1\}|$, hence $\sup_{i} |\{j: a_{ij}=1\}| < \infty$.

Suppose now that $A$ is of bounded degree. Then the transpose $A^{T}$ satisfies $\sup_{j} |\{i\in \mathbb{N}: a_{ij}=1\}| < \infty$. It means that $\sup_{i}\|A^{T} e_i\|_{1}<\infty$. Since the unit vectors $e_i$ are extreme points of the unit ball of $\ell^1$ it follows that $A^{T}$ extends to a bounded linear map on $\ell^1$. By dualizing we get that $A$ defines a weak$^{\ast}$ continuous (i.e. normal) map on $\ell^{\infty}$ and (complete) positivity is clear.
\end{proof}

\begin{defn}
A normal, completely positive map $A: \op{B} \to \op{B}$ is called a \textbf{quantum adjacency matrix of bounded degree} if for each $\alpha \in I$ it satisfies
\[
\frac{1}{\op{Tr}(\rho_{\alpha}^{-1})} \sum_{l} A(e_{il}^{\alpha}\rho_{\alpha}^{-1})A(e_{lj}^{\alpha}) = A(e_{ij}^{\alpha}).
\]
\end{defn}

\begin{prop}
Quantum adjacency matrices of bounded degree $A: \op{B} \to \op{B}$ correspond precisely to projections $P\in \op{B}\overline{\otimes} \op{B}^{\op{op}}$ that are \emph{integrable with respect to the second variable}, i.e. satisfy $(\op{Id} \otimes \psi^{\op{op}})P \in \op{B}$.
\end{prop}
\begin{proof}
If we start from $A$, we can use the formula $P:=\sum_{\alpha} \frac{1}{\op{Tr}(\rho_{\alpha}^{-1})} \sum_{i,j} A(\rho_{\alpha}^{-\frac{1}{4}}e_{ij}^{\alpha}\rho_{\alpha}^{-\frac{1}{4}}) \otimes (\rho_{\alpha}^{-\frac{1}{4}}e_{ji}^{\alpha} \rho_{\alpha}^{-\frac{1}{4}})^{\op{op}}$ to conclude that $(\op{Id} \otimes \psi^{\op{op}})P = \sum_{\alpha} A(\mathds{1}_{\alpha}) = A\mathds{1}$, because $A$ was assumed to be normal.

Assume now that $D:=(\op{Id}\otimes \psi^{\op{op}})P \in \op{B}$. Recall that in the finite dimensional case $A$ can be recovered from $P$ via the formula $Ax:=(\op{Id}\otimes \psi)(P\#(1\otimes \sigma_{\frac{i}{2}}(x)))$ and we just have to check that this formula still makes sense. It suffices to show that $(\varphi \otimes \psi)(P\#(1\otimes \sigma_{\frac{i}{2}}(x)))<\infty$ for any normal state $\varphi$ and a positive operator $x\in \op{B}$. To this end, we will denote by $x_m$ the projection of $x$ onto the finite direct sum $\bigoplus_{\alpha=1}^{m} M_{n_{\alpha}}$ and show that $\sup_{m} (\varphi \otimes \psi)(P\# (1\otimes \sigma_{\frac{i}{2}}(x_m)))$ is finite; $z_{m}$ will be the central projection corresponding to $\bigoplus_{\alpha=1}^{m} M_{n_{\alpha}}$. Because the right leg of our tensor product is now finite dimensional, we can write $P_m:= P (1\otimes z_{m})\leqslant P$ as a finite sum $P_m: =\sum_{i} P_{i} \otimes Q_{i}^{\op{op}}$ and our formula becomes $\sum_{i} \varphi(P_{i})\psi(\sigma_{\frac{i}{2}}(x_m) Q_{i})$. But $P_m$ is an orthogonal projection, so we actually get $\sum_{i,j} \varphi(P_{i}^{\ast} P_{j}) \psi(\sigma_{\frac{i}{2}}(x_{m}) Q_j Q_{i}^{\ast})$. By the KMS property of $\psi$ we get
\[
\psi(\sigma_{\frac{i}{2}}(x_{m}) Q_j Q_{i}^{\ast}) = \psi ((\sigma_{-\frac{i}{2}}(Q_{i}))^{\ast} x_{m} \sigma_{-\frac{i}{2}}(Q_{j})). 
\]
It means that we apply $\varphi\otimes \psi$ to the positive operator $R^{\ast} (1\otimes x_m) R$, where $R = \sum_{i} P_{i} \otimes \sigma_{-\frac{i}{2}}(Q_{i})$; it is bounded above by $\|x_m\|R^{\ast}R$. Note that 
\[
(\op{Id}\otimes \psi)(R^{\ast}R) = \sum_{i,j} P_{i}^{\ast}P_{j}\psi (\sigma_{\frac{i}{2}}(Q_{i}^{\ast}) \sigma_{-\frac{i}{2}}(Q_{j})) = \sum_{i,j} P_{i}^{\ast}P_{j} \psi(Q_{j} Q_{i}^{\ast}) = (\op{Id} \otimes \psi^{\op{op}})(P_{m}) \leqslant D,
\]
where once again we used the KMS property. It follows that $\varphi(A(x_m)) \leqslant \|x_m\|\varphi(D)$, thus we can define $A$ for any positive $x$; $A$ is a well-defined normal, completely positive map on $\op{B}$.
\end{proof}
There is another interesting class of infinite graphs that we can easily generalize to the quantum setting, namely the locally finite graphs. Classically, these are graphs whose adjacency matrix $A$ has finitely supported rows. The rows of the adjacency matrix are images of the standard basis vectors under the transpose $A^{T}$, i.e. locally finite graphs are exactly the ones, for which the transpose of the adjacency matrix preserves the subspaces of finitely supported functions. Recall from Proposition \ref{Prop:Generalizedchoi} that the transpose in the context of quantum adjacency matrices is exactly the KMS adjoint.
\begin{defn}
We say that the quantum graph defined by a quantum adjacency matrix $A: c_{00}(\op{B}) \to \op{B}$ is \textbf{locally finite} if $A^{\ast}_{\op{KMS}}(c_{00}(\op{B})) \subset c_{00}(\op{B})$ or, equivalently, the corresponding projection satisfies $x\cdot P \in c_{00}(\op{B}) \otimes c_{00}(\op{B}^{\op{op}})$ for any $x\in c_{00}(\op{B})$.
\end{defn}
We will now check that quantum graphs of bounded degree are locally finite.
\begin{prop}\label{Prop:locallyfinite}
If $A:\op{B}\to \op{B}$ is of bounded degree then it is locally finite.
\end{prop}
\begin{proof}
Let $P\in \op{B}\overline{\otimes} \op{B}^{\op{op}}$ be the corresponding projection. Recall that $P = (P_{\beta\alpha})_{\beta,\alpha}$, where $P_{\beta\alpha}$ is a projection in $M_{n_{\beta}}\otimes M_{n_{\alpha}}^{\op{op}}$. Define $P_{\beta}:= \sum_{\alpha} P_{\beta\alpha}$. Because $A$ is of bounded degree, we clearly have $(\psi \otimes \psi^{\op{op}})(P_{\beta}) < \infty$. It is easy to check that the weight $\psi \otimes \psi^{\op{op}}$ on $\op{B}\otimes \op{B}^{\op{op}}$ satisfies $mm^{\ast}=\op{Id}$. Therefore we get $(\psi \otimes \psi^{\op{op}})(P_{\beta}) \geqslant |\{\alpha: P_{\beta\alpha}\neq 0\}|$ by Lemma \ref{Lem:weightprojection}. It follows that for each $\beta$ there is only a finite number of $\alpha$'s such that $P_{\beta\alpha}\neq 0$, which means that $A$ is locally finite.
\end{proof}
Before we describe the bimodules that give rise to quantum graphs of bounded degree, note that from the formula \eqref{Eq:Adjacencykraus} (the formula is stated in the infinite dimensional case) it follows that the degree matrix $D = \sum_{i} V_{i} V_{i}^{\ast}$, where $V_{i}$ is an orthonormal basis of the $\op{B}'$-module $\sigma_{\frac{i}{4}}^{\psi^{-1}}(V)$ with respect to the inner product induced by $\psi^{-1}$. It is standard to check that the expression $\sum_{i} V_{i} V_{i}^{\ast}$ does not depend on the choice of the basis. We are now ready to state the first result.
\begin{prop}
If the $\op{B}'$-bimodule $V\subset \op{B}(\HH)$ defines a quantum graph of bounded degree then $\sigma_{\frac{i}{4}}^{\psi^{-1}}(V)$ is a self-dual Hilbert module over $\op{B}'$.
\end{prop}
\begin{proof}
We work in a more general setup where $\mathbb{E}:\mathsf{M} \to \mathsf{N}$ is a normal conditional expectation and $V \subset \mathsf{M}$ is a weak$^{\ast}$ closed right $\mathsf{N}$-module, so $\op{B}(\HH)$ is replaced by $\mathsf{M}$, $\op{B}'$ by $\mathsf{N}$, $\psi^{-1}$ by $\mathbb{E}$, and $\sigma_{\frac{i}{4}}^{\psi^{-1}}$ by $V$. We will follow the proof of \cite[Th\'{e}or\`{e}me 3.5]{MR945550}. We want to show that if there is an orthonormal family $(V_{j})_{j} \subset V$ such that $(V_{j})_{j}$ is an orthonormal basis inside the self-dual completion $V_{E}$ of $V$ such that $\sum_{j} V_{j} V_{j}^{\ast}$ converges weak$^{\ast}$ in $\mathsf{M}$ then $V$ is self-dual. It is sufficient to check that $V$ is complete with respect to the norm induced by the $\mathsf{N}$-valued inner product, because self-duality will follow, as $V$ admits a predual that makes the inner product separately weak$^{\ast}$ continuous. This is equivalent to the existence of a constant $K>0$ such that $K\|\mathbb{E}(v^{\ast}v)\|\geqslant \|v\|^2$ for all $v\in V$. Let $V_{E} \otimes_{\mathsf{N}} \mathsf{M}$ denote the self-dual $\mathsf{M}$-module, obtained as the completion of the relative tensor product. Then $\eta:= \sum_{j} V_{j}\otimes V_{j}^{\ast}$ is a well defined element of $V_{E} \otimes_{\mathsf{N}} \mathsf{M}$ by \cite[1.8(vi)]{MR945550}. We now define a completely positive map $F$ from $\op{B}_{\mathsf{N}}(V_{E})$ to $\mathsf{M}$ via $F(T):= \langle \eta, T\cdot \eta\rangle$. For any $v,w\in V$ we have 
\[
F(|v\rangle\langle w|)=\sum_{j,k}  \langle V_{j}^{\ast}, \langle V_{j}, v\rangle_{\mathsf{N}} \langle w, V_{k}\rangle_{N} V_{k}^{\ast}\rangle_{\mathsf{M}} = \sum_{j,k} V_{j} \langle V_{j}, v\rangle_{\mathsf{N}} \langle w, V_{k}\rangle_{\mathsf{N}} V_{k}^{\ast},
\]
which is equal to $vw^{\ast}$ as $(V_{j})_{j}$ was chosen to be an orthonormal basis. Since $\|F\|\leqslant \|\eta\|^2$, it follows that $\|v\|^2 = \|F(|v\rangle\langle v|)\|\leqslant \|\eta\|^2 \|\langle v, v\rangle\|= \|\eta\|^2 \|\mathbb{E}(v^{\ast}v)\|$. 
\end{proof}
If we want to specify a class of $\op{B}'$-bimodules that describe quantum graphs of bounded degree, then we need a property that does not depend on the choice of the embedding $\op{B} \subset \op{B}(\HH)$. This way we arrive at the following equivalence.
\begin{prop}
A quantum graph is of bounded degree if and only if for any faithful representation $\op{B}\subset \op{B}(\HH)$ the corresponding bimodule $\sigma_{\frac{i}{4}}^{\psi^{-1}}(V)$ is a self-dual Hilbert module over $\op{B}'$.
\end{prop}
\begin{proof}
We proved one implication in the previous proposition. Once again we work in a greater generality. Namely we assume that we have a faithful normal conditional expectation $\mathbb{E}: \mathsf{M} \to \mathsf{N}$ and $V \subset \mathsf{M}$ is a weak$^{\ast}$ closed self-dual $\mathsf{N}$-module. Taking advantage of the independence of the representation, we assume moreover that $V\overline{\otimes} \op{B}(\ell^{2})$ (weak$^{\ast}$ closure inside $\mathsf{M}\overline{\otimes} \op{B}(\ell^2)$) is a self-dual module over $\mathsf{N}\overline{\otimes} \op{B}(\ell^2)$. In the original formulation the module $V\overline{\otimes} \op{B}(\ell^{2})$ will correspond to the inflation of the representation of $\op{B}$, i.e. taking $\op{B} \subset \op{B}(\HH\otimes \ell^{2})$ instead of $\op{B} \subset \op{B}(\HH)$, hence will be self-dual by the assumption of the proposition.

Note that $V$  has two natural operator space structures, one inherited from $\mathsf{M}$, and the other one coming from the structure of an $\mathsf{N}$-module. The assumption that $V$ is self-dual, hence in particular complete, implies that the two Banach space structures are isomorphic, as the operator norm clearly dominates the Hilbert module norm, hence one can invoke the open mapping theorem. Our assumption that $V\overline{\otimes} \op{B}(\ell^{2})$ is also self-dual implies that even the operator space structures are isomorphic.

Let $(v_{1},\dots, v_{n})$ be an $n$-tuple in $V$ and view it as the first row of a matrix in $M_n(V)$. The square of the norm of this matrix, when we use the Hilbert module operator space structure, is equal to the norm of the matrix $(\mathbb{E}(v_{i}^{\ast} v_{j}))$ in $M_{n}(\mathsf{N})$, which is clearly the same as the norm in $M_{n}(\mathsf{M})$. This can be computed using the following formula (for positive matrices):
\[
\|(m_{ij})\|_{M_{n}(\mathsf{M})} = \sup\{\|\sum a_{i}^{\ast} m_{ij} a_{j}\|_{\mathsf{M}}: a_{i} \in \mathsf{M}, \|\sum a_{i}^{\ast} a_{i}\|_{\mathsf{M}}\leqslant 1\}.
\]
This formula is proved using the fact that the $C^{\ast}$-norm on $M_{n}(\mathsf{M})$ is the same as the operator norm, where $M_{n}(\mathsf{M})$ acts on the Hilbert $\mathsf{M}$-module $\mathsf{M}^{n}$. It follows from the equivalence of the operator space structures that for some $K>0$ we have the inequality
\[
\|\sum_{i,j} a_{i}^{\ast} v_{i}^{\ast} v_{j} a_{j}\| \leqslant K \|\sum_{i,j} a_{i}^{\ast} \mathbb{E}(v_{i}^{\ast} v_{j}) a_{j}\|.
\]
Once again we will follow the proof of \cite[Th\'{e}or\`{e}me 3.5]{MR945550}. The estimate above shows that we can define a bounded map $\tau: V \otimes_{\mathsf{N}} \mathsf{M} \to \mathsf{M}$ given by $\sum_{i} v_{i} \otimes m_{i}\mapsto \sum_{i} v_{i} m_{i}$. Since it is right $\mathsf{M}$-linear and $V\otimes_{\mathsf{N}} \mathsf{M}$ is self-dual, $\tau$ is given by an inner product with a vector $\eta \in V \otimes_{\mathsf{N}} \mathsf{M}$. Using $\eta$, we can define a normal, completely positive map $F: \op{B}_{\mathsf{N}}(V) \to \mathsf{M}$ via $F(T) := \langle \eta, T\cdot \eta\rangle$. Let now $(X_{k})_{k}$ be an orthonormal basis of $V$ and define $V_{n}$ to be the self-dual submodule of $V$ generated by $X_{1},\dots, X_{n}$. The corresponding map $\tau_{n}$ is clearly represented by the vector $\eta_{n}:= \sum_{k=1}^{n} X_{k} \otimes X_{k}^{\ast}$. Therefore one can compute the value of $F_{n}$ on a rank one operator $|v\rangle \langle w| \in \op{B}_{\mathsf{N}}(V_{n})$ as 
\[
F_{n}(|v\rangle\langle w|) = \sum_{j,k=1}^{n} X_{j} \langle X_{j}, w\rangle \langle w, X_{k}\rangle X_{k}^{\ast} = vw^{\ast}. 
\]
As $\op{Id}_{V_{n}} = \sum_{k=1}^{n} |X_{k}\rangle\langle X_{k}|$, we get $F_{n}(\op{Id}) = \sum_{k=1}^{n} X_{k}X_{k}^{\ast}$, therefore $\|\sum_{k=1}^{n} X_{k} X_{k}^{\ast}\|\leqslant K$. It follows that the partial sums $\sum_{k=1}^{n} X_{k} X_{k}^{\ast}$ form an increasing and bounded sequence, hence the series $\sum_{k=1}^{\infty} X_{k} X_{k}^{\ast}$ converges strongly to an operator $D$ with norm at most $K$, which is precisely the bounded degree condition.
\end{proof}
\begin{rem}
It is not clear whether one can simply assume that one of the bimodules defining the quantum graph is self-dual, i.e. whether the equivalence of the operator space structures follows automatically from the Banach space equivalence. For the index of a conditional expectation this is exactly the gap between \cite[Proposition 3.3]{MR945550} and \cite[Th\'{e}or\`{e}me 3.5]{MR945550}. However, this problem has been resolved in \cite{MR1642530}, where the authors proved that if for a conditional expectation $\mathbb{E}:\op{A} \to \op{B}$  there exists a constant $K>0$ such that the map $K \mathbb{E} - \op{Id}$ is positive then there also exists a constant $L>0$ such that $L\mathbb{E} - \op{Id}$ is \emph{completely} positive. Unfortunately the proof relies heavily on the $C^{\ast}$-algebraic structure and it is unclear whether an analogous result holds in our situation.
\end{rem}

\subsection{Explicit correspondence between the three approaches}\label{Subsec:Explicit}
From now on we assume that $\op{B} \simeq \ell^{\infty}-\bigoplus_{\alpha} M_{n_{\alpha}}$ is represented on $\HH:= \bigoplus_{\alpha} \C^{n_{\alpha}}$. In this case the commutant of $\op{B}$ is equal to the center of $\op{B}$. Therefore each weak$^{\ast}$ closed $\op{B}'$-bimodule is exactly a collection of subspaces $V_{\alpha\beta} \subset \op{B}(\C^{n_{\alpha}}, \C^{n_{\beta}})$. The conditional expectation $\psi^{-1}: \op{B}(\HH) \to \op{B}'$ is given by $\psi^{-1}(x) = \sum_{\alpha} \frac{\op{Tr}(\rho_{\alpha}^{-1} \mathds{1}_{\alpha} x \mathds{1}_{\alpha})}{\op{Tr}(\rho_{\alpha}^{-1})} \mathds{1}_{\alpha}$, where $\mathbb{1}_{\alpha} \in \op{B}^{\prime}$ is the minimal central projection corresponding to the summand $M_{n_{\alpha}}$. For each pair $(\alpha, \beta)$ let $\left(X_{i}^{\alpha\beta}\right)_{i=1}^{n_{\alpha\beta}}$ be an orthonormal basis of $V_{\alpha\beta}$ with respect to the KMS inner product on $\op{B}(\HH)$ (where we use the density $\bigoplus_{\alpha} \frac{1}{\op{Tr}(\rho_{\alpha}^{-1})} \rho_{\alpha}^{-1}$ to define a weight on $\op{B}(\HH)$), i.e. $\frac{1}{\sqrt{\op{Tr}(\rho_{\alpha}^{-1})\op{Tr}(\rho_{\beta}^{-1})}} \op{Tr}((X_{i}^{\alpha\beta})^{\ast} \rho_{\beta}^{-\frac{1}{2}} X_{j}^{\alpha\beta} \rho_{\alpha}^{-\frac{1}{2}}) = \delta_{ij}$. Orthogonal projection from $\op{B}(\C^{n_{\alpha}}, \C^{n_{\beta}})$ onto $V_{\alpha\beta}$ will be given by 
\begin{equation}\label{Eq:Projectionbimod}
P_{V_{\alpha\beta}}(x) = \frac{1}{\sqrt{\op{Tr}(\rho_{\alpha}^{-1})\op{Tr}(\rho_{\beta}^{-1})}} \sum_{i} X_{i}^{\alpha\beta}  \op{Tr}((X_{i}^{\alpha\beta})^{\ast} \rho_{\beta}^{-\frac{1}{2}} x \rho_{\alpha}^{-\frac{1}{2}}).
\end{equation}
We will use the notation $e_{kl}^{\alpha\beta}$ to denote the rank one operator $|e_{k}^{\alpha}\rangle\langle e_{l}^{\beta}|$ to rewrite it further as
\[
\frac{1}{\sqrt{\op{Tr}(\rho_{\alpha}^{-1})\op{Tr}(\rho_{\beta}^{-1})}}\sum_{i,k,l} X_{i}^{\alpha\beta} e_{kl}^{\alpha\beta} \rho_{\beta}^{-\frac{1}{2}} x \rho_{\alpha}^{-\frac{1}{2}} (X_{i}^{\alpha\beta})^{\ast} e_{lk}^{\beta\alpha}. 
\]
Recall from Proposition \ref{Prop:Bimodproj} (and the discussion preceding it) that to find the representing element in $M_{n_{\beta}} \otimes M_{n_{\alpha}}^{\op{op}}$ we need to twist by the modular group. 
\begin{prop}
Let  $\left(X_{i}^{\alpha\beta}\right)_{i=1}^{n_{\alpha\beta}}$ be an orthonormal basis of  $V_{\alpha\beta} \subset \op{B}(\C^{n_{\alpha}}, \C^{n_{\beta}})$ with respect to the KMS inner product induced by $\psi^{-1}$. Then the corresponding projection $P_{\beta\alpha} \in M_{n_{\beta}} \otimes M_{n_{\alpha}}^{\op{op}} $ is given by  
\begin{equation}
\frac{1}{\sqrt{\op{Tr}(\rho_{\alpha}^{-1})\op{Tr}(\rho_{\beta}^{-1})}}\sum_{i,k,l} \rho_{\beta}^{-\frac{1}{4}} X_{i}^{\alpha\beta} e_{kl}^{\alpha\beta} \rho_{\beta}^{-\frac{1}{4}} \otimes \left(\rho_{\alpha}^{-\frac{1}{4}} (X_{i}^{\alpha\beta})^{\ast} e_{lk}^{\beta\alpha} \rho_{\alpha}^{-\frac{1}{4}}\right)^{\op{op}}.
\end{equation}
In particular the flip $\sigma(P_{\beta\alpha}) \in M_{n_{\alpha}} \otimes M_{n_{\beta}}^{\op{op}}$ gives the orthogonal projection onto $V_{\alpha\beta}^{\ast} \subset \op{B}(\C^{n_{\beta}}, \C^{n_{\alpha}})$.
\end{prop}
\begin{proof}
The flip $\sigma(P_{\beta\alpha})$ is the projection corresponding to the orthonormal set $\left((X_{i}^{\alpha\beta})^{\ast}\right)_{i=1}^{n_{\alpha\beta}}$, which is an orthonormal basis of $V_{\alpha\beta}^{\ast}$, because of the property $\langle X,Y\rangle_{\op{KMS}} = \langle Y^{\ast}, X^{\ast}\rangle_{\op{KMS}}$ of the KMS inner product.
\end{proof}

To find the explicit formula for the associated adjacency matrix, we will rewrite the formula \eqref{Eq:Projectionbimod} in a different manner
\[
P_{V_{\alpha\beta}}(x) = \frac{1}{\sqrt{\op{Tr}(\rho_{\alpha}^{-1})\op{Tr}(\rho_{\beta}^{-1})}} \sum_{i,k,l} X_{i}^{\alpha\beta} e_{kl}^{\alpha} (X_{i}^{\alpha\beta})^{\ast} \rho_{\beta}^{-\frac{1}{2}} x \rho_{\alpha}^{-\frac{1}{2}} e_{lk}^{\alpha}.
\]
Therefore the corresponding projection in $\op{B}\overline{\otimes} \op{B}^{\op{op}}$ will be given by
\[
P_{\beta\alpha} = \frac{1}{\sqrt{\op{Tr}(\rho_{\alpha}^{-1})\op{Tr}(\rho_{\beta}^{-1})}} \sum_{i,k,l} \rho_{\beta}^{-\frac{1}{4}} X_{i}^{\alpha\beta} e_{kl}^{\alpha} (X_{i}^{\alpha\beta})^{\ast} \rho_{\beta}^{-\frac{1}{4}} \otimes \left(\rho_{\alpha}^{-\frac{1}{4}} e_{lk}^{\alpha}\rho_{\alpha}^{-\frac{1}{4}}\right)^{\op{op}}. 
\]
If we compare it with \eqref{Eq:ProjfromAdj}, we immediately obtain the following statement.
\begin{prop}\label{Prop:Adjacencykraus}
Let  $\left(X_{i}^{\alpha\beta}\right)_{i=1}^{n_{\alpha\beta}}$ be an orthonormal basis of  $V_{\alpha\beta} \subset \op{B}(\C^{n_{\alpha}}, \C^{n_{\beta}})$ with respect to the KMS inner product induced by $\psi^{-1}$. Then the corresponding quantum adjacency matrix $A: c_{00}(\op{B}) \to \op{B}$ is given by
\begin{equation}\label{Eq:Adjacencykraus}
A(e_{kl}^{\alpha}) = \sum_{\beta} \sqrt{\frac{\op{Tr}(\rho_{\alpha}^{-1})}{\op{Tr}(\rho_{\beta}^{-1})}}\sum_{i} \rho_{\beta}^{-\frac{1}{4}} X_{i}^{\alpha\beta} \rho_{\alpha}^{\frac{1}{4}} e_{kl}^{\alpha} \rho_{\alpha}^{\frac{1}{4}} (X_{i}^{\alpha\beta})^{\ast} \rho_{\beta}^{-\frac{1}{4}}.
\end{equation} 
\end{prop}
\begin{rem}
We could in principle get rid of the scalars $\op{Tr}(\rho_{\alpha}^{-1})$ just by normalizing our densities so that these numbers are equal to $1$. However, in the context of quantum groups it is more common to choose the normalization $\op{Tr}(\rho_{\alpha}^{-1}) = \op{Tr}(\rho_{\alpha})$, so we decided to include a more general discussion.
\end{rem}

\section{Covariant quantum adjacency matrices}\label{Sec:covqadj}
Let $\G$ be a compact quantum group and let $\bGamma$ be its discrete dual. We want to find quantum adjacency matrices on $\ell^{\infty}(\bGamma) \simeq \ell^{\infty}-\bigoplus_{\alpha \in \op{Irr}\G} M_{n_{\alpha}}$ that are covariant with respect to the right action of $\bGamma$ on itself. Such an operator $T$ commutes with all right convolution operators, so $T(x) = T(\delta_{\varepsilon} \ast x) = T(\delta_{\varepsilon})\ast x$, where $\delta_{\varepsilon}$ is the unit with respect to the convolution, namely the unit of the block corresponding to the trivial representation of $\G$. It follows that it is given by a left convolution\footnote{This sweeps under the rug some potential analytical difficulties but it only serves as a motivation here, so we feel that there is no need to delve deeper into this issue.} and we just need to figure out, which ones give rise to quantum adjacency matrices. Before that we need to gather some information about convolutions on quantum groups.
\subsection{Fourier transform on compact/discrete quantum groups}
Here we collect some conventions about the Fourier transform on compact and discrete quantum groups. We will always denote by $\G$ a compact quantum group and by $\bGamma$ its discrete dual.  

First, let $\varphi$ be a bounded functional on $L^{\infty}(\mathbb{G})$. We will define its Fourier transform to be an element $\widehat{\varphi} \in \ell^{\infty}(\bGamma)\simeq\ell^{\infty}-\bigoplus_{\alpha \in \op{Irr}(\mathbb{G})} \op{M}_{n_{\alpha}}$ such that
\[
\widehat{\varphi}(\alpha):= (\op{Id}\otimes \varphi)\left((u^{\alpha})^{\ast}\right).
\]
We define the convolution of two functionals as
\[
\varphi_1 \ast \varphi_2 := (\varphi_1 \otimes \varphi_2)\circ \Delta
\]
With this convention it is easy to check that $(\varphi_1 \ast \varphi_2)\hat{\phantom{i}}(\pi) =  \widehat{\varphi_2}(\pi) \widehat{\varphi_1}(\pi)$. 

Recall (see \cite[Sections 1.3 and 1.4]{MR3204665}) that for any $\alpha \in \op{Irr}(\mathbb{G})$ there is a unique positive matrix $\rho_{\alpha} \in M_{n_{\alpha}}$ such that $\op{Tr}(\rho_{\alpha}) = \op{Tr}(\rho_{\alpha}^{-1})=:\op{dim}_{q}(\alpha)$ and it implements an equivalence between the representation $\alpha$ and its double contragredient. These matrices can be used to describe the left and right Haar measures on the discrete dual $\bGamma$:
\begin{align*}
\widehat{h_{R}}(x) &= \sum_{\alpha} \op{Tr}(\rho_{\alpha}) \op{Tr}(\rho_{\alpha} x_{\alpha}) \\
\widehat{h_{L}}(x) &= \sum_{\alpha} \op{Tr}(\rho_{\alpha}) \op{Tr}(\rho_{\alpha}^{-1} x_{\alpha}).
\end{align*}
Both $\widehat{h_{R}}$ and $\widehat{h_{L}}$ satisfy $mm^{\ast}=\op{Id}$ by Proposition \ref{Prop:mstar}.

To define the Fourier transform and convolution of elements of $L^{\infty}(\mathbb{G})$ we will always associate to such an element $x \in L^{\infty}(\mathbb{G})$ a linear functional $h(\cdot x)$, where $h$ is the Haar state on $\G$. For further use, we will compute now the Fourier transform of an element $x:=\sum_{i,j=1}^{n_{\alpha}} x_{ij} u_{ij}^{\alpha}$ belonging to the span of matrix elements of a fixed irreducible representation $\alpha$.
\begin{lem}\label{Lem:FTmatrixelement}
Let $x:=\sum_{i,j=1}^{n_{\alpha}} x_{ij} u_{ij}^{\alpha}$. We have $\widehat{x}(\pi)= \delta_{\alpha\pi}\frac{1}{\dim_q(\alpha)} X^{\op{T}} \rho_{\alpha}^{-1}$, where $(X)_{ij} = x_{ij}$.
\end{lem}
\begin{proof}
We have $\widehat{x}(\pi) = \sum_{i,j,k,l} x_{ij} h( \left(u_{lk}^{\pi}\right)^{\ast} u_{ij}^{\alpha}) e_{kl}$. Using the orthogonality relations $h(\left(u_{lk}^{\pi}\right)^{\ast} u_{ij}^{\alpha}) = \delta_{\alpha \pi}\frac{1}{\dim_q(\alpha)}\delta_{jk} \left(\rho_{\alpha}^{-1}\right)_{il}$, we easily find that 
\[
\widehat{x}(\pi) = \frac{\delta_{\alpha\pi}}{\dim_q(\alpha)} \sum_{i,j,l} x_{ij} \left(\rho_{\alpha}^{-1}\right)_{il} e_{jl} = \frac{\delta_{\alpha\pi}}{\dim_q(\alpha)} X^{\op{T}}\rho_{\alpha}^{-1}.
\]
\end{proof}
In the special case of a single matrix element we get the formula $(u_{ij}^{\alpha})\hat{\phantom{a}}(\pi) = \frac{\delta_{\pi \alpha}}{\dim_q(\alpha)} e_{ji}^{\alpha} \rho_{\alpha}^{-1}$. Given Proposition \ref{Prop:mstar}, it suggests that we will be working with the Haar measure $\widehat{h_{R}}$ on $\bGamma$.

\begin{lem}
The Fourier transform extends to a unitary between $L^{2}(\mathbb{G})$ and $\ell^{2}(\bGamma)$, where $\ell^{2}(\bGamma)$ is defined using the right Haar measure on $\bGamma$.
\end{lem}
\begin{proof}
Using orthogonality relations and the previous lemma one can check that $\langle u_{ij}^{\alpha}, u_{kl}^{\pi}\rangle = \langle \widehat{u_{ij}^{\alpha}}, \widehat{u_{kl}^{\pi}}\rangle$. Therefore the Fourier transform is isometric on $Pol(\mathbb{G})$, thus it extends to an isometry on $L^{2}(\mathbb{G})$, which is also surjective, as the image is clearly dense.
\end{proof}

We will also need the fact that the inverse Fourier transform maps convolution into the regular product. We will redefine the Fourier transform using the right multiplicative unitary $W$, defined via $W(a\otimes b) = \Delta(a)(1\otimes b)$ on $L^{\infty}(\G)\otimes L^{\infty}(\G)$ and extended to $L^{2}(\G)\otimes L^{2}(\G)$ by continuity; it corresponds to the right regular representation of $\G$ and implements the multiplication via $W(x\otimes \mathds{1})W^{\ast} = \Delta(x)$. It can be shown that $W \in \ell^{\infty}(\bGamma)\overline{\otimes} L^{\infty}(\G)$. It also satisfies the pentagonal equation $W_{23} W_{12} = W_{12} W_{13} W_{23}$, which amounts to coassociativity of the comultiplication. We can write it more explicitly as $W= \sum_{\alpha \in \op{Irr}(\G)} \sum_{i,j}^{n_{\alpha}} e_{ij}^{\alpha} \otimes u_{ij}^{\alpha}$, where $e_{ij}^{\alpha}$ acts on $L^{2}(\G)$ via right convolution by its inverse Fourier transform. Because of this formula, we can also compute the Fourier transform as $\widehat{\varphi}= (\op{Id} \otimes \varphi)(W^{\ast})$. We will now use the multiplicative unitary to define the Fourier transform on $\bGamma$.
\begin{defn} Let $\psi$ be a bounded functional on $\ell^{\infty}(\bGamma)$ then its Fourier transform will be defined by $\widehat{\psi}:= (\psi \otimes \op{Id})(W) \in L^{\infty}(\mathbb{G})$; the fact that we use the same notation for Fourier transforms going both ways should not cause any confusion. For any $x \in c_{00}(\bGamma)$ we can define a functional on $\ell^{\infty}(\bGamma)$ by $\widehat{h_{R}}(\cdot x)$ and define the Fourier transform of $x$ using this embedding. 
\end{defn}
\begin{lem}
We have $\widehat{e_{ji}^{\alpha} \rho_{\alpha}^{-1}} = \op{dim}_q(\alpha) u_{ij}^{\alpha}$ and if $\psi_{1},\psi_{2} \in \left(\ell^{\infty}(\bGamma)\right)^{\ast}$ then $\widehat{\psi_{1}\ast \psi_{2}} = \widehat{\psi_{2}} \widehat{\psi_{1}}$.
\end{lem}
\begin{proof}
We have
\[
\widehat{e_{ji}^{\alpha} \rho_{\alpha}^{-1}} = \sum_{\pi,k,l} \widehat{h_{R}}(e_{kl}^{\pi} e_{ji}^{\alpha}\rho_{\alpha}^{-1}) u_{kl}^{\alpha},
\]
so $\pi=\alpha$ and $l=j$. We can further rewrite our formula as
\[
\sum_{k}\op{Tr}(\rho_{\alpha}) \op{Tr}(\rho_{\alpha}e_{ki}^{\alpha}\rho_{\alpha}^{-1}) u_{kj}^{\alpha} = \op{Tr}(\rho_{\alpha}) u_{ij}^{\alpha} = \op{dim}_{q}(\alpha) u_{ij}^{\alpha}.
\]
In particular, the Fourier transform on $\bGamma$ is inverse to the Fourier transform on $\mathbb{G}$ (e.g. on the level of $L^{2}$-spaces). 

Let us check that the convolution is mapped into the usual product, using the pentagonal equation. To this end, let $\psi_{1}$ and $\psi_2$ be two functionals on $\ell^{\infty}(\bGamma)$. We have $\widehat{\psi_1 \ast \psi_2} = (\psi_1\ast \psi_2 \otimes \op{Id})(W)$. By definition of the convolution, it can be rewritten as 
\[
(\psi_1 \otimes \psi_2 \otimes \op{Id})(\Delta_{\bGamma} \otimes \op{Id})(W).
\]
The comultiplication in $\bGamma$ is defined in such a way that $(\Delta_{\bGamma}\otimes \op{Id})(W) = W_{23}W_{13}$. Therefore we have to compute $(\psi_1 \otimes \psi_2 \otimes \op{Id})(W_{23} W_{13})$. It is clear that $\psi_2$ acts on the part $W_{23}$ and $\psi_1$ acts on $W_{13}$, hence we get $\widehat{\psi_2} \widehat{\psi_1}$ as the result.
\end{proof}
We will now investigate the behaviour of the Fourier transform under the antipode. Recall that on any locally compact quantum group the antipode $S$ admits a polar decomposition $S = \tau_{-\frac{i}{2}} R$, where $(\tau_{t})_{t\in\mathbb{R}}$ is the so-called scaling group (hence $\tau_{-\frac{i}{2}}$ is only defined on analytic elements) and $R$ is the unitary antipode, which is an involutive $\ast$-antiautomorphism; we also have $\widehat{h_{R}}\circ R = \widehat{h_{L}}$. 
\begin{lem}[{\cite[Proposition 6.8 and Proposition 7.9]{MR1832993}}]
Let $\bGamma$ be a locally compact quantum group. Let $(\sigma_{t})_{t\in\mathbb{R}}$ and $(\sigma^{'}_t)_{t\in \mathbb{R}}$ be the modular groups of the right and left Haar measures , and let $(\tau_{t})_{t\in\mathbb{R}}$ be the scaling group. Then the following commutation relations hold:
\begin{enumerate}[{\normalfont (i)}]
\item $\Delta \circ \sigma_t = (\sigma_t \otimes \tau_{-t})\circ \Delta$;
\item $\Delta \circ \sigma_t^{'} = (\tau_{t} \otimes \sigma_{t}^{'})\circ \Delta$;
\item $\Delta \circ \tau_t = (\tau_t \otimes \tau_t)\circ \Delta$;
\item $\Delta \circ \tau_t = (\sigma_t^{'}\otimes \sigma_{-t}) \circ \Delta$.
\end{enumerate}
If $\bGamma$ is discrete then $\sigma_{t}^{'} = \sigma_{-t}$, $(\sigma_t \otimes \sigma_t)\circ \Delta = \Delta \circ \sigma_t$, and $\tau_t = \sigma_{-t}$.
\end{lem}
\begin{proof}
We will only explain how to get the additional statements about discrete quantum groups. The fact that $\sigma_{t}' = \sigma_{-t}$ follows immediately from the explicit formulas for the Haar measures. The modular element of a quantum group is relating the left and right Haar measures and in the discrete case it is equal to $\delta:= \bigoplus_{\alpha \in \op{Irr}(\mathbb{G})} \rho_{\alpha}^2$. It always satisfies $\Delta(\delta) = \delta\circ \delta$, which implies that $\Delta \circ \sigma_t = (\sigma_t \otimes \sigma_t) \circ \Delta$. Now $\tau_{t} =\sigma_{-t}$ follows easily from the equalities $\Delta \circ \tau_{t} = (\sigma_{-t}\otimes \sigma_{-t})\circ \Delta$ and $(\sigma_{-t}\otimes \sigma_{-t})\circ \Delta = \Delta \circ \sigma_{-t}$, as they imply $\Delta\circ \tau_{t} = \Delta \circ \sigma_{-t}$ and $\Delta$ is injective.
\end{proof}
\begin{lem}[{\cite[Lemma 3.3]{MR2669427}}]
Let $\psi \in \ell^{\infty}(\bGamma)^{\ast}$ be such that $\psi^{\#}(x):= \psi^{\ast}(S^{-1}(x)) = \overline{\psi(S(x^{\ast}))}$ is also bounded. Then $(\widehat{\psi})^{\ast} = \widehat{\psi^{\#}}$.
\end{lem}
\begin{proof}
See the proof of \cite[Lemma 3.3]{MR2669427}, but note that our convention for Fourier transform uses the adjoint of the multiplicative unitary, that is why we need to work with $S^{-1}$ rather than $S$. We also used the formula $(S^{-1}(x))^{\ast} = S(x^{\ast})$.
\end{proof}
We will now see, what happens when the functional comes from an element of $c_{00}(\bGamma)$.
\begin{lem}\label{Lem:antipode}
Let $x \in c_{00}(\bGamma)$ and $\psi:= \widehat{h_{R}}(\cdot x)$. Then $\psi^{\#}= \widehat{h_{R}}(\cdot S(x^{\ast}) \rho^{-2})$.
\end{lem}
\begin{proof}
We have to compute 
\[
\psi^{\#}(y) = \overline{\widehat{h_{R}}(S(y^{\ast})^{\ast} x))} = \widehat{h_{R}}(x^{\ast} S(y^{\ast})^{\ast}) = \widehat{h_{R}}(x^{\ast} S^{-1}(y)).
\]
Since $S= \tau_{-\frac{i}{2}} R$ and $\tau_{-\frac{i}{2}} = \sigma_{\frac{i}{2}}$, we get
\[
\widehat{h_{R}}(R(\sigma_{-\frac{i}{2}}(y)  R(x^{\ast}))). 
\]
Since $\widehat{h_{R}}\circ R = \widehat{h_{L}}$ and $\widehat{h_{L}}(y) = \widehat{h_{R}}(\rho^{-2} y)$, we obtain
\[
\widehat{h_{R}}(\rho^{-2} \sigma_{-\frac{i}{2}}(y) R(x^{\ast})).
\]
Using the invariance under the modular group, and the fact that $\rho^{-2}$ is in the centralizer, we arrive at
\[
\widehat{h_{R}}(y \sigma_{\frac{i}{2}} (R(x^{\ast}) \rho^{-2}).
\]
As $\sigma_{\frac{i}{2}} R(x^{\ast}) = S(x^{\ast})$, we get our answer:
\[
\psi^{\#}(y) = \widehat{h_{R}}(y S(x^{\ast}) \rho^{-2}).
\]
\end{proof}
\subsection{Quantum adjacency matrices as convolution operators}
We work with the von Neumann algebra $\ell^{\infty}(\bGamma)$ equipped with the weight $\widehat{h_{R}})$. We will now assume that our adjacency matrix is given by a convolution against an element $P$, where we view it as a functional $\widehat{h_{R}}(\cdot P)$. To avoid technical issues, we will only convolve against elements belonging to the algebraic direct sum of matrix algebras. This is also justified by the fact that $P\ast \mathds{1} = \widetilde{h_{R}}(P)\mathds{1}$ and if this number is finite and $P$ is a projection (see Proposition \ref{Prop:projectionadj}), then $P\in c_{00}(\bGamma)$ by Lemma \ref{Lem:weightprojection}. Suppose then that our quantum adjacency matrix is given by the following formula:
\[
Ax:= P \ast x.
\]
We will prove a general formula for a Schur product of two convolution operators.
\begin{prop}\label{Prop:Schurproductconvolution}
Let $P_{1}, P_{2} \in c_{00}(\bGamma)$ and define $A_{1}(x):= P_{1}\ast x$, $A_{2}(x):= P_{2}\ast x$. Then $A:=m(A_{1}\otimes A_{2})m^{\ast}$ is given by $Ax:= P_1P_2 \ast x$.
\end{prop}
\begin{proof}
In this proof we will use the notation $\cF$ for the Fourier transform on $\mathbb{G}$.

The formula for $A$ is (for each $\alpha \in \op{Irr}\G$):
\[
\frac{1}{\op{dim}_q(\alpha)} \sum_{k} A_1(e_{ik}^{\alpha}\rho_{\alpha}^{-1}) A_2(e_{kj}^{\alpha}\rho_{\alpha}^{-1}) = A(e_{ij}^{\alpha}\rho_{\alpha}^{-1}).
\]
By Lemma \ref{Lem:FTmatrixelement} we can write $e_{ij}^{\alpha}\rho_{\alpha}^{-1} = \op{dim}_q(\alpha) \cF(u_{ji}^{\alpha})$, so we want
\[
\sum_{k} A_1(\cF(u_{ki}^{\alpha})) A_2(\cF(u_{jk}^{\alpha})) = A(\cF(u_{ji}^{\alpha})).
\]
As the Fourier transform switches convolution to multiplication (and reverses the order), we have $A_k\circ \cF = \cF \circ \widehat{A_k}$, where $\widehat{A_k}y:= y \widehat{P_k}$, with $\widehat{P_k}$ being the inverse Fourier transform of $P_k$. We can therefore rewrite our condition further as
\[
\sum_{k} \cF( u_{ki}^{\alpha} \widehat{P_1}) \cF( u_{jk}^{\alpha} \widehat{P_2}) = \cF ( u_{ji}^{\alpha}\widehat{P}),
\]
and using once again the properties of the Fourier transform we finally arrive at 
\[
\sum_{k}  u_{jk}^{\alpha}\widehat{P_2} \ast  u_{ki}^{\alpha} \widehat{P_1} =  u_{ji}^{\alpha} \widehat{P}.
\]
As the span of $u_{ij}$'s is dense, we can rewrite this equation as (using Sweedler's notation)
\[
  x_{(1)}\widehat{P_2} \ast  x_{(2)} \widehat{P_1} =  x \widehat{P}.
  \]
For $x=\mathds{1}$ we get $\widehat{P_2} \ast \widehat{P_1} = \widehat{P}$, i.e. $P = P_1 P_2$. This condition turns out to be sufficient as well. Indeed, we can compute $ x_{(1)} \widehat{P_2} \ast  x_{(2)}\widehat{P_1}$ acting on an element $b$ as
\begin{align*}
( x_{(1)}\widehat{P_2} \ast  x_{(2)}\widehat{P_1})(b) &= (h(\cdot x_{(1)} \widehat{P_2}) \otimes h(\cdot x_{(2)}\widehat{P_1}) \Delta(b) \\
&= (h\otimes h)(\Delta(b)(x_{(1)} \otimes x_{(2)})(\widehat{P_2} \otimes \widehat{P_1})) \\
&= (h\otimes h)(\Delta(b)\Delta(x) (\widehat{P_2} \otimes \widehat{P_1}) ) \\
&= (h\otimes h)( \Delta(bx)(\widehat{P_2} \otimes \widehat{P_1})) \\ 
&= (\widehat{P_2} \ast \widehat{P_1})(bx) = \widehat{P}(bx), 
\end{align*}
where the last expression is clearly equal to $x\widehat{P}$ acting on $b$. 
\end{proof}
\begin{prop}\label{Prop:projectionadj}
The map $A: \ell^{\infty}(\bGamma) \to \ell^{\infty}(\bGamma)$ given by $Ax:= P\ast x$ with $P\in c_{00}(\bGamma)$ is a quantum Schur idempotent iff $P^2 = P$ is an idempotent. Moreover, it is completely positive iff $P$ is a projection.
\end{prop}
\begin{proof}
We have already showed the first part. For the second part, we will check that $(P\ast x)^{\ast} = P^{\ast}\ast x^{\ast}$. Indeed, let $P\ast x = T$, which means that for any $y\in \ell^{\infty}(\bGamma)$ we have $\widehat{h_{R}}(y_{(1)}P) \widehat{h_{R}}(y_{(2)} x) = \widehat{h_{R}}(yT)$. We then have
\[
\widehat{h_{R}}(yT^{\ast}) = \overline{\widehat{h_{R}}(T y^{\ast})} = \overline{\widehat{h_{R}}(\sigma_{i}(y^{\ast}) T)}.
\]
We can now use the defining property of $T$ and the fact that on a discrete quantum group the comultiplication is a $\ast$-homomorphism that intertwines the modular group and arrive at
\[
\widehat{h_{R}}(\sigma_{i}(y^{\ast}) T) = \widehat{h_{R}}(\sigma_{i}(y_{(1)}^{\ast}) P) \widehat{h_{R}}(\sigma_{i}(y_{(2)}^{\ast}) x).
\]
After applying the complex conjugation we obtain
\[
\overline{\widehat{h_{R}}(\sigma_{i}(y^{\ast}) T)} = \widehat{h_{R}}(P^{\ast}\sigma_{-i}(y_{(1)}) )\widehat{h_{R}}(x^{\ast}\sigma_{-i}(y_{(2)})),
\]
which is equal to
\[
\widehat{h_{R}}(y_{(1)}P^{\ast} )\widehat{h_{R}}(y_{(2)}x^{\ast} ) = (P^{\ast}\ast x^{\ast})(y),
\]
thus $(P \ast x)^{\ast} = P^{\ast} \ast x^{\ast}$, hence $A$ is $\ast$-preserving iff $P=P^{\ast}$, and, as we already mentioned in Remark \ref{Rem:starvspositive}, complete positivity is equivalent to being $\ast$-preserving for quantum Schur idempotents.
\end{proof}
We will now characterize which covariant adjacency matrices are GNS-symmetric and which are KMS-symmetric.
\begin{thm}\label{Thm:KMSsymmetry}
Let $P=P^{\ast} \in c_{00}(\bGamma)$. $Ax:= P\ast x$ is GNS-symmetric iff $P=S(P)$ and $A$ is KMS-symmetric iff $P=R(P)$.
\end{thm}
\begin{proof}
We will view the convolution operator as a multiplication operator on the dual compact quantum group, using the Fourier transform. After taking the Fourier transform $A$ will become a \emph{right} multiplication operator by $\widehat{P}$. The adjoint of a right multiplication operator by $x$ is the right multiplication by $\sigma_{-i}^{\mathbb{G}}(x^{\ast})$, as can be checked using the KMS property. Therefore the GNS-symmetry amounts to $(\widehat{P})^{\ast} = \sigma_{i}(\widehat{P})$. We know from Lemma \ref{Lem:antipode} that $(\widehat{P})^{\ast}$ is the Fourier transform of $S(P^{\ast})\rho^{-2}$. Let us now check that $\sigma_{z}(\widehat{P})$ is the Fourier transform of $\rho^{iz} P \rho^{iz}$. Indeed, it is sufficient to check the formula when $\widehat{P} = u_{kl}^{\alpha}$, and we can moreover assume that we have chosen a basis such that $\rho_{\alpha}$ is diagonal with diagonal entries $\rho_{\alpha,k}$. Then $\sigma_{z}(u_{kl}^{\alpha}) = \rho_{\alpha, k}^{iz} \rho_{\alpha, l}^{iz} u_{kl}^{\alpha}$, which is the Fourier transform of $ \frac{1}{\op{Tr}(\rho_{\alpha}^{-1})}\rho_{\alpha, l}^{iz} e_{lk}^{\alpha} \rho_{\alpha, k}^{iz} \rho_{\alpha}^{-1}$, proving the formula. We therefore arrive at the equality
\[
S(P^{\ast})\rho^{-2} = \rho^{-1} P \rho^{-1},
\]
i.e. $ S(P^{\ast}) = \sigma_{i}(P)$. Recall that $S^{2} = \tau_{-i} = \sigma_{i}$, hence $S(P^{\ast}) = S^{2}(P)$, i.e. $S(P) = P^{\ast}$. Since by assumption $P=P^{\ast}$, we obtain $S(P)=P$.

For KMS-symmetry, note that it is equivalent to the GNS-symmetry of the map $\widetilde{A}(x):= \sigma_{-\frac{i}{4}} A(\sigma_{\frac{i}{4}}(x))$. It is not difficult to check that this map is given by $\widetilde{A}x = \sigma_{-\frac{i}{4}}(P)\ast x$, using the fact that the comultiplication intertwines the modular group in our case. We therefore get $S(\sigma_{-\frac{i}{4}}(P)) = (\sigma_{-\frac{i}{4}}(P))^{\ast} = \sigma_{\frac{i}{4}}(P^{\ast})$. This leads to $S \sigma_{-\frac{i}{2}}(P) = P^{\ast}$, but $S \sigma_{-\frac{i}{2}} = S \tau_{\frac{i}{2}} = R$, so $R(P)=P^{\ast}$. By assumption $P=P^{\ast}$, thus $R(P)=P$.
\end{proof}
\section{Quantum Cayley graphs}\label{Sec:qcayley}
Classically the Cayley graph is defined using a generating set of our group. In the quantum case we will replace a finite subset of a group with a projection $P \in c_{00}(\bGamma)$. Typically Cayley graphs are undirected, so we will also assume that $P$ is invariant under the unitary antipode $R$. Moreover, a generating set should not contain the unit, which will easily translate to the fact that the counit $\varepsilon$ applied to $P$ should be equal to zero; less abstractly, the component of $P$ corresponding to the trivial representation of $\mathbb{G}$ should be equal to $0$. The only nontrivial condition is what it means for such a projection to be generating. We will denote by $[T]$ the range projection of an operator $T$.
\begin{defn}
Let $\bGamma$ be a discrete quantum group and let $P\in c_{00}(\bGamma)$ be a projection such that $R(P)=P$ and $\varepsilon(P) = 0$. We say that $P$ is \textbf{generating} if $\bigvee_{n} [P^{\ast n}] = \mathds{1} \in \ell^{\infty}(\bGamma)$. Then we define the quantum Cayley graph of $\bGamma$ via the adjacency matrix
\[
Ax: = P \ast x.
\]
These quantum graphs are always \textbf{regular} with $D:= A\mathds{1} = \widehat{h_{R}}(P) \mathds{1}$.
\end{defn}
\begin{rem}
The condition $\varepsilon(P)=0$ is equivalent to the fact that the resulting quantum graph has no loops, i.e. $m(A\otimes \op{Id})m^{\ast}$. Indeed $\op{Id}$ is a convolution operator against $\delta_{\varepsilon}$, so $m(A\otimes \op{Id})m^{\ast}$ is a convolution operator against $P\delta_{\varepsilon} = \varepsilon(P)\delta_{\varepsilon}$ by Proposition \ref{Prop:Schurproductconvolution}.
\end{rem}
There is a natural special case, where we choose a central generating projection, which leads to a more familiar notion.
\begin{prop}\label{Prop:generatingrep}
Let $z \in c_{00}(\bGamma)$ be a central projection, i.e. it is a sum of minimal central projections corresponding to certain irreducible representations of $\mathbb{G}$: $z = \bigoplus_{k=1}^{n} \mathds{1}_{\alpha_{k}}$. Let $S=\{\alpha_1,\dots, \alpha_n\}$. Then $z= R(z)$ iff $S=\overline{S}$, i.e. it is closed under taking conjugate representations. The condition $\varepsilon(z)=0$ is equivalent to the fact that $S$ does not contain the trivial representation. Moreover, $z$ is generating iff the set $S$ is a generating set of irreducible representations of $\mathbb{G}$.
\end{prop}
\begin{proof}
Only the statement about the criterion for being a generating set requires proof. Let $\mathds{1}_{n_{\alpha}}$ be a minimal central projection. By Lemma \ref{Lem:FTmatrixelement} we have that $\widehat{\mathds{1}_{n_{\alpha}} \rho_{\alpha}^{-1}} =\op{Tr}(\rho_{\alpha}^{-1}) \chi_{\alpha}$. Because of the formula $\sigma_{w}^{\mathbb{G}}(\widehat{T}) = \rho^{iw} T \rho^{iw}$ established in the proof of Theorem \ref{Thm:KMSsymmetry}, we get $\widehat{\mathds{1}_{n_{\alpha}}} = \op{Tr}(\rho_{\alpha}^{-1}) \sigma^{\mathbb{G}}_{-\frac{i}{2}}(\chi_{\alpha})$. Using the Fourier transform, we can view the convolution operator by $z$ as a multiplication operator by the character (acted on by the modular group) of the representation $\pi:=\alpha_1\oplus\dots\oplus \alpha_n$. If $S$ is a generating set of representations, then eventually all characters will appear in the decomposition of the powers of the character $\chi_{\pi}$. By taking an inverse Fourier transform, we will obtain that $z^{\ast m}$ is a linear combination of central projections and eventually all minimal central projections will appear in its decomposition, which shows that $z$ is generating. If we assume that $z$ is generating we can reverse this argument to show that $S$ is a generating set of representations.
\end{proof}
\begin{rem}
The condition that $\G$ admits a finite generating set of representations is used in \cite{Fim10} as a definition of $\bGamma$ being finitely generated. This is equivalent to existence of a generating projection belonging to $c_{00}(\bGamma)$. Indeed, a generating set of representations directly gives a (central) projection in $c_{00}(\bGamma)$ and if we start from any generating projection then its central cover remains generating and an element of $c_{00}(\bGamma)$.
\end{rem}
\begin{ex}
We already mentioned that if $P \in c_{00}(\bGamma)$ is a generating projection, then so is its central cover $z(P)$. The converse, however is not true, as the following simple example shows. Let $G=SU(2)$ and let $\bGamma$ be its dual. We will take as $P$ the rank one projection $e_{11}$ inside the $M_2$ summand corresponding to the fundamental representation $\pi$ of $SU(2)$. Its central cover corresponds to a generating representation, hence is generating itself; we will show that $P$ is not generating. Convolution powers of $e_{11}$ are not difficult to compute using the Fourier transform. Indeed, the Fourier transform of $e_{11}$ is a multiple of the matrix element $u_{11}$, the matrix element of the fundamental representation $\pi$ at the vector $e_{1}$. The power $u_{11}^{n}$ is the matrix element of $\pi^{\otimes n}$ at the vector $e_{1}^{\otimes n}$. This vector happens to be a symmetric tensor, so $u_{11}^{n}$ is a matrix element of the symmetric tensor power $\op{Sym}^{n}(\pi)$. It is well-known that this representation is irreducible, hence the inverse Fourier transform of $u_{11}^{n}$ is a multiple of a rank one projection inside the summand $M_{n+1}$. It is clear that $P$ is not generating, because it misses a lot, e.g. the trivial representation. 
\end{ex}
\subsection{Independence of the choice of the generating set}
For a group $\Gamma$ the Cayley graph does depend on the choice of the generating set but they are all bi-Lipschitz equivalent, so more or less indistinguishable in the sense of metric/coarse geometry. Fortunately there is a metric structure available in the quantum case that will allow us to define the notion of bi-Lipschitz equivalence. It is a slightly
\begin{defn}[Kuperberg-Weaver, \cite{MR2908248}]
Let $\mathsf{M} \subset \op{B}(\HH)$ be a von Neumann algebra. A \textbf{quantum metric} on $\mathsf{M}$ is an increasing family of weak$^{\ast}$ closed operator systems $(V_{t})_{t\geqslant 0}$ such that
\begin{enumerate}[(i)]
\item $V_{0} = \mathsf{M}'$,
\item $V_s V_t \subset V_{s+t}$,
\item $V_t = \bigcap_{s>t} V_{s}$.
\end{enumerate}
\end{defn}
For $\mathsf{M} = \ell^{\infty}(X)$ we can intepret these operator systems $V_t$ as operators of propagation at most $t$. 
\begin{rem}
This notion of a quantum metric is related to the more familiar one due to Rieffel (see \cite{Rie04}). Namely, a quantum metric in the sense of Kuperberg and Weaver yields a Lipschitz seminorm in the sense of Rieffel (see \cite[Definition 4.19]{MR2908248}).
\end{rem}

For \emph{undirected} quantum graphs the quantum metric can be built in the following way. First of all, we can associate a weak$^{\ast}$ closed $\mathsf{M}'$-bimodule $V\subset \op{B}(\HH)$ to a quantum adjacency matrix, using a representation $\mathsf{M}\subset \op{B}(\HH)$; just like in Subsection \ref{Subsec:Explicit} we we will work with the representation $\mathsf{M}\simeq \ell^{\infty}-\bigoplus_{\alpha} M_{n_{\alpha}} \subset \op{B}(\bigoplus_{\alpha} \mathbb{C}^{n_{\alpha}})$. Then we simply define $V_{t} := \overline{\op{span}}\{\sum_{k=0}^{[t]} V^{k}\}$, where we interpret $V^{0}$ as $\mathsf{M}'$. 
\begin{defn}
Let $(V_{t})_{t\geqslant 0}$ and $(W_{t})_{t\geqslant 0}$ be two quantum metrics on $\mathsf{M}$. We say that they are \textbf{bi-Lipschitz equivalent} if there exists a constant $M>0$ such that $V_{t} \subset W_{Mt}$ and $W_{t} \subset V_{Mt}$ for all $t\geqslant 0$.
\end{defn}
Therefore to understand the metric structure of a quantum graph, we need a better understanding of the powers $V^k$. We will relate them in some way to the powers of the adjacency matrix, so we need to be able to compose the adjacency matrix with itself, which is possible for quantum graphs of bounded degree\footnote{It is also possible for locally finite quantum graphs, but then the Choi matrix associated with the power of the adjacency matrix is only \emph{affiliated} with $\op{B}\overline{\otimes} \op{B}^{\op{op}}$ and we decided to avoid these technicalities.}.
\begin{lem}\label{lem:poweradj}
Let $A: \op{B} \to \op{B}$ be a quantum adjacency matrix of bounded degree, let $P \in \op{B}\overline{\otimes} \op{B}^{\op{op}}$ be its associated projection, and let $V$ be the corresponding $\op{B}'$ bimodule. Let $A_k \in \op{B}\overline{\otimes} \op{B}^{\op{op}}$ be the generalized Choi matrix of $A^k$. Then the range projection $[A_k]$ corresponds to the $\op{B}'$-bimodule $V^k$.
\end{lem}
\begin{proof}
Since $\op{B} \overline{\otimes} \op{B}^{\op{op}} \simeq \ell^{\infty}-\bigoplus_{\alpha, \beta} M_{n_{\beta}} \otimes M_{n_{\alpha}}^{\op{op}}$, it suffices to check that the range projection $[A_{k}]_{\beta\alpha}$ corresponds to $V^{k}_{\alpha\beta} \subset \op{B}(\C^{n_{\alpha}}, \C^{n_{\beta}})$. 

We will first work in the tracial setting, and then discuss how to modify the proof in the general case. We are in the following situation: we have a completely positive map $\Phi(x):= \sum_{k=1}^{n} V_{k} x V_{k}^{\ast}$ and we would like to check that the support projection of the Choi matrix of $\Phi$ gives rise to the orthogonal projection onto the span of $V_{k}$'s inside $\op{B}(\C^{n_{\alpha}}, \C^{n_{\beta}})$; indeed, the Kraus operators of $A^{k}_{\alpha\beta}$ are the products belonging to $V^{k}_{\alpha\beta}$ (see \eqref{Eq:Adjacencykraus} and note that the sums appearing will be finite by Proposition \ref{Prop:locallyfinite}). 

Let $T$ be the Choi matrix of $\Phi$. As a positive semidefinite matrix it can be diagonalized and written in the form $T = \frac{1}{n_{\alpha}}\sum_{l=1}^{d} \lambda_{l} T_{l} e_{ij}^{\alpha} T_{l}^{\ast} \otimes (e_{ji}^{\alpha})^{\op{op}}$, where $T_{l}: \C^{n_{\alpha}}\to \C^{n_{\beta}}$ satisfy $\frac{1}{n_{\alpha}} \op{Tr} (T_{k}^{\ast} T_{l}) = \delta_{kl}$, $d$ is the rank of $T$, and $\lambda_l$'s are the non-zero eigenvalues. The support projection of $T$ is equal to $[T]= \frac{1}{n_{\alpha}}\sum_{l=1}^{d} T_{l} e_{ij}^{\alpha} T_{l}^{\ast} \otimes (e_{ji}^{\alpha})^{\op{op}}$, i.e. by setting all the non-zero eigenvalues to be equal to $1$. The corresponding subspace $V_{\alpha\beta}$ is the span of the $T_{l}$'s. Luckily it is a well-known fact that the span of Kraus operators is independent of any choices hence $\op{span}\{T_{l}: l \in [d]\} = \op{span}\{V_k: k \in [n]\}$. This finishes the proof of the lemma.

Let us comment now, what changes in the non-tracial setting. Combining Lemma \ref{Lem:choirep} with Lemma \ref{Lem:modularidem} we see that the generalised Choi matrix  of $\Phi$ is equal to $\frac{1}{\op{Tr}(\rho_{\alpha}^{-1})} \sum_{i,j} V_{k} \rho_{\alpha}^{-\frac{1}{2}} e_{ij}^{\alpha} \rho_{\alpha}^{-\frac{1}{2}} V_{k}^{\ast} \otimes e_{ji}^{\alpha}$. Once again we diagonalize it to obtain $(T_{1},\dots, T_{d})$ with $\frac{1}{n_{\alpha}} \op{Tr}(T_{k}^{\ast} T_{l}) = \delta_{kl}$. Define $\widetilde{T_{l}}:= \frac{\sqrt{\op{Tr}(\rho_{\alpha}^{-1})}}{\sqrt{n_{\alpha}}} T_{l} \rho_{\alpha}^{\frac{1}{2}}$. Then the generalised Choi matrix is equal to 
\[
\frac{1}{\op{Tr}(\rho_{\alpha}^{-1})} \sum_{l,i,j} \lambda_{l} \widetilde{T_{l}} \rho_{\alpha}^{-\frac{1}{2}} e_{ij}^{\alpha} \rho_{\alpha}^{-\frac{1}{2}} \widetilde{T_{l}}^{\ast} \otimes e_{ji}^{\alpha},
\]
where $\frac{1}{\op{Tr}(\rho_{\alpha}^{-1})} \op{Tr} (\rho_{\alpha}^{-1} \widetilde{T_k}^{\ast} \widetilde{T_l}) = \delta_{kl}$, from which it follows that the support projection is once again obtained by setting all $\lambda_l$'s to $1$.
\end{proof}
\begin{thm}\label{Thm:bilipschitz}
Let $\bGamma$ be a discrete quantum group and let $(\bGamma, P_1)$ and $(\bGamma, P_2)$ be its two quantum Cayley graphs. Then the corresponding quantum metric spaces are bi-Lipschitz equivalent.
\end{thm}
\begin{proof}
Let $A_1 x:= P_1 \ast x$ and $A_2 x:= P_2 \ast x$. Note that the powers $A_1^k$ and $A_2^k$ can be computed using the convolution powers $P_1^{\ast k}$ and $P_2^{\ast k}$. From Proposition \ref{Prop:Schurproductconvolution} it follows that the support projection of the Choi matrix of $A_{i}^{k}$ is the Choi matrix of the convolution operator by $[P_{i}^{\ast k}]$. Since both projections are generating, there exists $M\in \mathbb{N}$ such that $P_{1}\leqslant \bigvee_{i=1}^{M} [P_{2}^{\ast i}]$ and $P_{2} \leqslant \bigvee_{i=1}^{M} [P_{1}^{\ast i}]$. The bi-Lipschitz equivalence follows easily.
\end{proof}
\subsection{Quantum Cayley graphs as quantum relations}
In this subsection we will describe the $\op{B}'$-bimodule corresponding to a covariant quantum adjacency matrix coming from a central projection, so $A: \op{B} \to \op{B}$ will be given by $Ax := \mathds{1}_{\gamma} \ast x$; for simplicity we will assume that $\gamma$ is irreducible.

In order to find $V_{\alpha\beta}$ we need to compute the part of the convolution $\mathbf{1}_{\gamma}\ast e_{ij}^{\alpha}$ belonging to $M_{n_{\beta}}$. We will assume that that the $\rho$ matrices are diagonal.
\begin{lem}
We have $(\mathbf{1}_{\gamma} \ast e_{ij}^{\alpha})_{\beta} =\frac{\op{dim}_{q}(\alpha)\op{dim}_{q}(\gamma)}{\op{dim}_q(\beta)} \sum_{k=1}^{n_{\gamma}}\sum_{l=1}^{m(\beta, \alpha\otimes \gamma)} V_{kl} e_{ij}^{\alpha} V_{kl}^{\ast}$, where $m(\beta,\alpha\otimes\gamma)$ is the multiplicity of $\beta$ inside $\alpha\otimes \gamma$, $V_{l}: \HH_{\beta} \to \HH_{\alpha}\otimes \HH_{\gamma}$ are morphisms, which are isometries with orthogonal ranges, and $V_{kl}(v):= V_{l}^{\ast}(v\otimes e_{k})$.
\end{lem}
\begin{proof}
By Lemma \ref{Lem:FTmatrixelement} we have
\[
\widehat{\mathbf{1}_{\gamma}} = \op{dim}_q(\gamma) \sum_{k=1}^{n_{\gamma}}\rho_{\gamma k} u_{kk}^{\gamma}
\]
and
\[
\widehat{e_{ij}^{\alpha}} = \op{dim}_q(\alpha) \rho_{\alpha j} u_{ji}^{\alpha}.
\]
What we need now is a formula for a product of two matrix coefficients and this is exactly \cite[Lemma 4.1]{MR3871830}:
\[
(u_{ji}^{\alpha} u_{kk}^{\gamma})_{\beta} = \sum_{l=1}^{m(\beta, \alpha\otimes \gamma)} \sum_{p,q} \langle e_{j} \otimes e_{k}, Ve_{p}\rangle u_{pq}^{\beta} \langle V e_{q}, e_{i}\otimes e_{k}\rangle.
\]
\end{proof}
Inverting the Fourier transform, we obtain 
\[
(\mathbf{1}_{\gamma}\ast e_{ij}^{\alpha})_{\beta} = \frac{\op{dim}_{q}(\alpha)\op{dim}_{q}(\gamma)}{\op{dim}_q(\beta)} \sum_{k=1}^{n_{\gamma}}\rho_{\gamma k} \rho_{\alpha j} \sum_{l=1}^{m(\beta, \alpha\otimes \gamma)} \sum_{p,q} \langle e_{j} \otimes e_{k}, V_l e_{p}\rangle e_{qp}^{\beta}\rho_{\beta p}^{-1} \langle V_l e_{q}, e_{i}\otimes e_{k}\rangle.
\]
Recall that $e_{qp}^{\beta}= |e_{q}\rangle\langle e_{p}|$, so using the equalities $\op{Id}_{\beta} = \sum_{p} |e_{p}\rangle \langle e_{p}|$, we arrive at
\[
(\mathbf{1}_{\gamma}\ast e_{ij}^{\alpha})_{\beta} = \frac{\op{dim}_{q}(\alpha)\op{dim}_{q}(\gamma)}{\op{dim}_q(\beta)} \sum_{k=1}^{n_{\gamma}}\rho_{\gamma k}\rho_{\alpha j}\sum_{l=1}^{m(\beta,\alpha\otimes \gamma)} |V_{l}^{\ast} (e_{i}\otimes e_{k})\rangle \langle V^{\ast}(e_{j}\otimes e_{k})| \rho_{\beta}^{-1}.
\]
We can further rewrite it as
\[
(\mathbf{1}_{\gamma}\ast e_{ij}^{\alpha})_{\beta} = \frac{\op{dim}_{q}(\alpha)\op{dim}_{q}(\gamma)}{\op{dim}_q(\beta)} \sum_{k=1}^{n_{\gamma}}\rho_{\gamma k}\sum_{l=1}^{m(\beta,\alpha\otimes\gamma)} V_{kl} e_{ij}^{\alpha} \rho_{\alpha} V_{kl}^{\ast} \rho_{\beta}^{-1}.
\]
The last thing to check is that $ V_{kl} \rho_{\gamma k} \rho_{\alpha} = \rho_{\beta} V_{kl}$. Recall $V_{kl}$ is a composition of two maps, an inclusion $\iota_{k}(v):= v\otimes e_{k}$ and a morphism $V_{l}^{\ast}: \HH_{\alpha}\otimes \HH_{\gamma} \to \HH_{\beta}$. It is clear that $\iota_{k} \rho_{\alpha} \rho_{\gamma k}= (\rho_{\alpha}\otimes \rho_{\gamma}) \iota_{k}$. It is not difficult to check that for two representations $X$ and $Y$ and an intertwiner $V \in \op{Mor}(X,Y)$ we have $V \rho_{X} = \rho_{Y} V$ because by Schur's lemma it is sufficient to check it when $X$ and $Y$ are multiples of the same irreducible representation $\alpha$ and for them the $\rho$ operator is merely a direct sum of copies of $\rho_{\alpha}$. By \cite[Theorem 1.4.9]{MR3204665} $\rho_{\alpha \otimes \gamma} = \rho_{\alpha} \otimes \rho_{\gamma}$, so we obtain $V_{l}^{\ast} (\rho_{\alpha}\otimes \rho_{\gamma}) = \rho_{\beta} V_{l}^{\ast}$. It follows that $ V_{l}^{\ast}\iota_{k} \rho_{\gamma k} \rho_{\alpha} = \rho_{\beta} V_{l}^{\ast}\iota_{k}$.
\begin{prop}
Let $\bGamma$ be a discrete quantum group and let $S$ be a generating set of representations of $\mathbb{G}$. Then the bimodule describing the quantum Cayley graph is given by
\[
V_{\alpha\beta} = \op{span}\{V_{s}^{\ast}(\iota_s(v)): s \in S, V_{s} \in \op{Mor}(\beta, \alpha\otimes s)\}, v\in \HH_{s}\}
\]
where $\iota_{s}(v): \HH_{\alpha} \to \HH_{\alpha} \otimes \HH_{s}$ is given by $\iota_{s}(v)(w):= w\otimes v$.
\end{prop}
\begin{proof}
From the previous lemma it follows that the span of the Kraus operators of the convolution operator against $\mathbf{1}_{\gamma}$, or rather the block mapping from $M_{n_{\alpha}}$ to $M_{n_{\beta}}$ is equal to
\[
\op{span}\{V^{\ast}(\iota(v)): V \in \op{Mor}(\beta, \alpha\otimes s)\}, v\in \HH_{\gamma}\}.
\] 
From the formula \ref{Eq:Adjacencykraus} it is clear that this span is equal to $(\sigma_{-\frac{i}{4}}^{\psi^{-1}}(V))_{\alpha\beta}$, rather than $V_{\alpha\beta}$. However, since we convolve against a central projection, the adjacency matrix commutes with the modular group, hence the corresponding bimodule will be invariant under the modular group, so $\sigma_{-\frac{i}{4}}^{\psi^{-1}}(V) = V$.

In general we convolve against a central projection corresponding to a direct sum of irreducible representations and we obtain the result by linearity.
\end{proof}
\subsection{A sample application: discrete quantum groups of subexponential growth are amenable}

Many algebraic properties of groups are encoded in the geometry of their Cayley graphs. Here we would like to present an indication that the same should hold for discrete \emph{quantum} groups. We will use the result from \cite{MR2502497} to show that subexponential growth of the Cayley graph implies amenability of the discrete quantum group (actually the result even shows that the dual compact quantum group is coamenable, which is formally a stronger property, but actually equivalent by \cite{MR2276175}).
\begin{defn}
Let $\bGamma$ be a discrete quantum group. We say that $\bGamma$ is of subexponential growth if there exists a generating projection $P \in c_{00}(\bGamma)$ with $\varepsilon(P)=1$ such that $\lim_{n\to\infty} \left( \widehat{h_{R}}([P^{\ast n}])\right)^{\frac{1}{n}} = 1$.
\end{defn}
\begin{rem}
We assume that the counit applied to $P$ is equal to $1$ because we want to measure the volume of the ball, so we want $[P^{\ast n}]$ to correspond to the set of words of length at most $n$, not exactly $n$.
\end{rem}
\begin{prop}
Suppose that $\bGamma$ is a discrete quantum group of subexponential growth. Then $\bGamma$ is amenable.
\end{prop}

\begin{proof}
First we explain why the notion of subexponential growth is independent of the choice of a generating projection. Take two generating projection $P$ and $Q$ and assume we have subexponential growth with respect to $P$. Call $V_{n}$ the quantum relations corresponding to $[P^{\ast n}]$ and $W_n$ the ones corresponding to $[Q^{\ast n}]$. By Theorem \ref{Thm:bilipschitz} there exists $M\in \mathbb{N}$ such that $ V_{n} \subset W_{Mn}$ and $W_{n} \subset V_{Mn}$. Since $\widehat{h_{R}}([P^{\ast n}])\mathds{1}$ is the degree matrix of $V_n$, we obtain 
\begin{align*}
\widehat{h_{R}}([P^{\ast n}]) &\leqslant \widehat{h_{R}}([Q^{\ast M n}]) \\
\widehat{h_{R}}([Q^{\ast n}]) &\leqslant \widehat{h_{R}}([P^{\ast M n}]).
\end{align*}
It follows that if $\lim_{n\to\infty} \left( \widehat{h_{R}}([P^{\ast n}])\right)^{\frac{1}{n}} = 1$ then $\lim_{n\to\infty} \left( \widehat{h_{R}}([Q^{\ast n}])\right)^{\frac{1}{n}} = 1$.

If $P\in c_{00}(\bGamma)$ is generating then its central cover $z(P) \in c_{00}(\bGamma)$ is generating as well, hence we can assume that we are working with a generating set $S$ of representations, including the trivial one (see Proposition \ref{Prop:generatingrep}). 

According to \cite[Theorem 3.3 and Theorem 4.5]{MR2502497} coamenability of $\mathbb{G}$ is equivalent to the following F{\o}lner condition: for every finitely supported, symmetric probability measure $\mu$ on $\op{Irr}(\mathbb{G})$ whose support contains the trivial representation and every $\varepsilon > 0$ there exists a finite subset $F \subset \op{Irr}(\mathbb{G})$ such that
\[
\sum_{\alpha \in \op{supp}(\mu \ast \mathbf{1}_{F})} n_{\alpha}^2 \leqslant (1+\varepsilon) \sum_{\alpha \in F} n_{\alpha}^2,
\]
where $n_{\alpha}$ is the dimension of the representation $\alpha$.

Note the inequality $n_{\alpha}^2 \leqslant \op{Tr}(\rho_{\alpha}) \op{Tr}(\rho_{\alpha}^{-1})$, from which it follows that for any central projection $z \in c_{00}(\bGamma)$ corresponding to a family $T$ of irreducible representations we have $\sum_{\alpha \in T} n_{\alpha}^2 \leqslant \widehat{h_{R}}(z)$.

Let $S^k$ denote the set of irreducible representations of $\mathbb{G}$ that appear in $k$-fold tensor products of elements of $S$. It follows that the sequence $a_k:= \sum_{\alpha \in S^{k}} n_{\alpha}^2$ satisfies $\lim_{k\to\infty} (a_{k})^{\frac{1}{k}} = 1$. 

Fix $\mu$ as above and $\varepsilon >0$. For a large enough $k$ we will have $\op{supp}(\mu) \subset S^{k}$. Clearly $\op{supp}(\mu \ast \mathbf{1}_{S^{n}}) \subset S^{k+n}$, so it suffices to find $n$ large enough so that $a_{k+n} \leqslant (1+\varepsilon) a_{n}$. As $\lim_{k\to\infty} (a_{k})^{\frac{1}{k}} = 1$ we have $\lim_{n\to\infty} \frac{a_{k+n}}{a_{n}} = 1$, so we are done.

\end{proof}
In the non-unimodular case (e.g. for $SU_q(2)$) it often happens that the classical dimensions exhibit subexponential growth, while the quantum ones grow exponentially, so Kyed's can still be applied. It is however not clear at this point how to encode such a growth condition in a natural way using the quantum Cayley graphs.

\subsection{Examples}
Our first example will be constructed using the duals of the free unitary quantum groups introduced in \cite{MR1382726}. Recall that if $F \in M_{n}$ is an invertible matrix then $C(U_{F}^{+})$ is the universal $C^{\ast}$-algebra generated by the entries of a unitary matrix $U:= (u_{ij})_{i,j\in [n]}$ such that $F\overline{U}F^{-1}$ is also unitary, where $\overline{U}:= (u_{ij}^{\ast})_{i,j\in [n]}$; the comultiplication is given by the familiar formula $\Delta(u_{ij}) := \sum_{k=1}^{n} u_{ik}\otimes u_{kj}$. In \cite[Th\'{e}or\`{e}me]{MR1484551} the representation theory has been computed and we will need the following two facts about it: all irreducibles are generated by the fundamental representation $u$ and its conjugate $\overline{u}$ and the tensor powers $u^{\otimes n}$ are irreducible.
\begin{ex}
Let $U_{F}^{+}$ be a free unitary quantum group. Let $P:=\mathbf{1}_{u} + \mathbf{1}_{\overline{u}}$ be the central projection in $\ell^{\infty}(\widehat{U_{F}^{+}})$ corresponding to the representation $u\oplus \overline{u}$. $P$ is generating and we call the associated quantum graph the quantum Cayley graph of $\widehat{U_{F}^{+}}$. 
\end{ex}
Here we would like to give some indication that this quantum Cayley graph should be in fact a quantum tree. For now we do not have a notion of a tree for quantum graphs, but we we would like to nonetheless argue why the Cayley graph of $\widehat{U_{F}^{+}}$ should be called a tree. We can split the quantum adjacency matrix $A$ into a sum of two convolution operators $A_{1}(x):= \mathbf{1}_{u} \ast x$ and $A_{2}(x):= \mathbf{1}_{\overline{u}} \ast x$. Because $u^{\otimes n}$ are $\overline{u}^{\otimes n}$ are irreducible for each $n\in \mathbb{N}$, hence the convolution powers are easy to compute using the Fourier transform; we have that $A_{1}^{n}(x) = \mathbf{1}_{u^{\otimes n}} \ast x$ and analogously for $A_2$. Classically this would mean that the direct graphs defined by $A_1$ and $A_2$ have the following property: for any pair of vertices there is at most one directed path connecting them. This property is weaker than being a tree, but provides some evidence that this example might be a tree. Moreover, we have $m(A_{1}\otimes A_{2})m^{\ast} = 0$, which means that $A_{1}$ and $A_{2}$ have disjoint sets of edges, and $A_{2}$ is the KMS adjoint of $A_{1}$. Combining all this information, we see that $A(x):= P\ast x$ is built as a sum of superposing two opposite orientations, so perhaps could be an example of  an unoriented quantum tree.

The next example concerns the famous compact quantum group $SU_q(2)$ introduced by Woronowicz. The representation theory is the same as for $SU(2)$, so we get a two-dimensional, self-conjugate, generating representation $\pi$.
\begin{ex}
Let $\pi$ be the fundamental representation of $SU_q(2)$ and let $P$ be the corresponding central projection in $\ell^{\infty}(\widehat{SU_q(2)})$. The associated quantum graph is the quantum Cayley graph of $\widehat{SU_q(2)}$.
\end{ex}
The quantum adjacency matrix is related in this case to the Markov operator of the random walk on the dual of $SU(2)$ considered by Biane in \cite{MR1109477}. After all, the quantum Cayley graphs are regular, hence $A$ just needs to be rescaled by a constant to give a random walk operator.

It is interesting to note that these quantum Cayley graphs are non-isomorphic for all $q\in (0,1]$. Indeed, any isomorphism would be a $\ast$-isomorphism of algebras preserving the appropriate weights and intertwining the adjacency matrices. In particular it would preserve the quantum dimensions of representations but these are different for different parameters.

As our last example we choose (the dual of) the free orthogonal quantum group $O_{F}^{+}$ (see \cite{MR1382726}), where $F \in M_{n}$ is an invertible matrix such that $F \overline{F} = c\mathds{1}$ for some $c\in \mathbb{R}$. $C(O_{F}^{+})$ is the universal $C^{\ast}$-algebra generated by the entries of a unitary matrix $U:= (u_{ij})_{i,j\in [n]}$ such that $U = F \overline{U} F^{-1}$. In \cite{MR1378260} it was proved that the representation theory is the same as for $SU(2)$: the fundamental representation is irreducible, all irreducible representations are indexed by the natural numbers and
\[
n \otimes m \simeq \bigoplus_{k=|n-m|}^{n+m} k.
\]

In this case we will take as the generating projection the central projection corresponding to the fundamental representation. 

In \cite{MR2130588} an alternative approach to Cayley graphs of discrete quantum groups has been developed, based on Hilbert spaces; there a notion of a tree is defined based on the classical graph built on the set of irreducible representations of the dual compact quantum group. The author proved (see \cite[Proposition 4.5]{MR2130588}) that one gets a tree iff this compact quantum group is a finite free product of free unitary and orthogonal quantum groups and a specific central projection is chosen -- the one corresponding to the sum of respective fundamental representations.

In our framework we do not yet have tools to prove that the quantum Cayley graph of $\widehat{O_{F}^{+}}$ is a tree thus it seems necessary to understand the precise relationship between our approach and that of Vergnioux. In particular it would be interesting to check whether there exists a reasonable notion of being a tree in our framework that can be detected using only classical graphs.
\subsection{Concluding remarks}
This paper should be viewed as a first step towards understanding the quantum graphs associated to discrete quantum groups. The motivation for this work has been twofold: to motivate the need for studying both infinite and non-tracial quantum graphs, but also to introduce a new geometric tool to aid in working with quantum groups. So far we have mostly developed a general framework and now more effort should be put into investigating the combinatorial/geometric properties of specific examples. It might suggest what to do for general quantum graphs, e.g. what are paths in a quantum graph. Paths are an indispensable tool in classical theory and would likely be equally important in our case. 

Another project would be to understand the link between property (T) of a discrete quantum group and the spectral gap of the random walk on its quantum Cayley graph. Maybe one can even mimic the famous construction of Margulis to obtain a nice sequence of quantum expanding graphs.

\section*{Acknowledgements}
The author is grateful to Adam Skalski, Christian Voigt, and the anonymous referee for many helpful remarks. The author was partially supported by the National Science Center, Poland (NCN) grant no. 2021/43/D/ST1/01446. The project is co-financed by the Polish National Agency for Academic Exchange within Polish Returns Programme. 

\vspace{5 pt}
\includegraphics[scale=0.5]{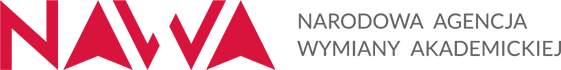}

\newcommand{\etalchar}[1]{$^{#1}$}

\end{document}